\documentclass{siamltex}
\usepackage{amsmath,amsfonts,amssymb}
\usepackage{latexsym}
\usepackage{epsfig,overpic}
\usepackage{subfigure}
\usepackage{color}

\newcommand{\R}{\mathbb R}
\newcommand{\bA}{\mathbf A}
\newcommand{\bB}{\mathbf B}

\newcommand{\bH}{\mathbf H}
\newcommand{\bI}{\mathbf I}

\newcommand{\bg}{\mathbf g}
\newcommand{\bn}{\mathbf n}
\newcommand{\bp}{\mathbf p}

\newcommand{\bu}{\mathbf u}

\newcommand{\bw}{\mathbf w}
\newcommand{\bx}{\mathbf x}

\newcommand{\Q}{\mathcal Q}

\newcommand{\T}{\mathcal T}

\newcommand{\Div}{\mathop{\rm div}}
\newcommand{\DivG}{{\operatorname{div_\Gamma}}}
\newcommand{\nablaG}{\nabla_\Gamma}

\newcommand{\uD}{\underline{D}}

\newcommand{\rd}{\mathrm{d}}

\newcommand{\Gs}{\mathcal{S}} 

\renewcommand{\div}{\textrm{div}}

\DeclareGraphicsExtensions{.pdf,.eps,.ps,.eps.gz,.ps.gz,.eps.Y}

\newcommand{\la}{\left\langle}
\newcommand{\ra}{\right\rangle}
\newcommand{\Wo}{\overset{\circ}{W}}

\newtheorem{assumption}{Assumption}[section]

\newtheorem{remark}{Remark}[section]
\def\enorm#1{|\!|\!| #1 |\!|\!|}
\begin{document}
\title{Error analysis of a space-time finite element method for solving PDEs on evolving surfaces
\thanks{Partially supported by NSF through the Division of Mathematical Sciences grant 1315993.}}
\author{Maxim A. Olshanskii\thanks{Department of Mathematics, University of Houston, Houston, Texas 77204-3008 (molshan@math.uh.edu).}
\and Arnold Reusken\thanks{Institut f\"ur Geometrie und Praktische  Mathematik, RWTH-Aachen
University, D-52056 Aachen, Germany (reusken@igpm.rwth-aachen.de).}
}

\maketitle

\begin{abstract} In this paper we present an error analysis of an Eulerian finite element method for
solving parabolic partial differential equations posed on evolving hypersurfaces in $\mathbb{R}^d$, $d=2,3$.
The  method employs discontinuous piecewise linear in time -- continuous piecewise linear in space finite elements and is based on a space-time weak formulation of a surface PDE problem. Trial and test surface finite element spaces
consist of traces of standard volumetric elements on a space-time manifold resulting from the evolution
of a surface. We prove first order convergence in space and time of the method in an energy norm and
second order  convergence in a weaker norm. Furthermore, we derive regularity results for solutions
of  parabolic PDEs on an evolving surface, which we need in a duality argument used in the proof of the second order convergence estimate.
\end{abstract}

\section{Introduction}

Partial differential equations posed on evolving surfaces appear in a number of  applications.  Well-known examples are
the diffusion and transport of  surfactants along interfaces in multiphase fluids \cite{GReusken2011,Stone}, diffusion-induced grain boundary motion \cite{GrainBnd1,GrainBnd2} and lipid interactions in moving cell membranes \cite{ElliotStinner,Novaketal}.
Recently, several numerical approaches for handling such type of problems have been introduced, cf.~\cite{DEreview}. In \cite{Dziuk07,DziukElliot2013a} Dziuk and Elliott developed and analyzed a finite element method for computing transport and diffusion on a surface which is based on a \emph{Lagrangian} tracking of the surface evolution. If a surface undergoes strong deformation, topological changes, or is defined implicitly, e.g., as the zero level of a level set function, then numerical methods based on a  Lagrangian approach have certain disadvantages. Methods using an \emph{Eulerian} approach were developed in e.g. \cite{DziukElliot2010,XuZh}, based on an extension  of the surface PDE into a bulk domain that contains the surface. An error analysis of this class of Eulerian methods for PDEs on an evolving surface is not known.

In the present paper, we analyze an Eulerian finite element method for parabolic type equations posed on evolving surfaces introduced in \cite{refJoerg,ORXsinum}. This method does not use an extension of the PDE off the surface into the bulk domain.
Instead, it uses restrictions of (usual) volumetric finite element functions to the surface, as first suggested in \cite{OlshReusken08,OlshanskiiReusken08} for stationary surfaces.
The method that we study uses continuous piecewise linear in space -- discontinuous piecewise linear in time volumetric finite element spaces. This allows a natural time-marching procedure, in which the numerical approximation is computed on one time slab after another. Moreover, spatial meshes may vary per time slab.  Therefore, in our surface finite element method  one can use  adaptive mesh refinement in space and time as explained in \cite{ErikJohn} for the heat equation in Euclidean space.
Numerical experiments in \cite{refJoerg,ORXsinum} have shown the efficiency of the approach  and demonstrated  second order accuracy of the method in space and time for problems with smoothly evolving surfaces. {In \cite{GOReccomas} a numerical example with two colliding spheres is considered, which illustrates the robustness of the method with respect to topological changes.}
 We consider this method to be a natural and effective extension of the approach from \cite{OlshReusken08,OlshanskiiReusken08} for \emph{stationary} surfaces to the case of \emph{evolving} surfaces. Until now, no error analysis of this (or any other) Euclidean finite element method for PDEs on evolving surfaces is known. In this paper we present such an error analysis.

The paper is organized as follows. In section~\ref{sec2}, we formulate the PDE that we consider on an evolving hypersurface in $\Bbb{R}^d$,  recall a weak formulation  and a corresponding well-posedness result. This weak formulation uses integration over the space-time manifold in  $\mathbb{R}^{d+1}$ and is well suited for our surface finite element method. This finite element method is explained in section~\ref{sec3}.  The error analysis starts with
a discrete stability result that is derived in section~\ref{sec4}. In Section~\ref{sec5},  a continuity estimate for
the  bilinear form is proved.  An error bound in a suitable energy norm is derived in  section~\ref{sec6}. The analysis has the same structure as in the standard Cea's lemma: a Galerkin orthogonality property is combined with continuity and discrete stability properties and with an interpolation error bound.  For the latter  we need  suitable extensions of a function defined on a space-time smooth manifold. The error bound in the energy norm guarantees first order convergence if spatial and time mesh sizes are of the same order.
In section~\ref{sec7}, we derive a second order error bound in a weaker norm. For this we use a duality argument and  need a higher order regularity estimate for the solution of a parabolic problem on a smoothly evolved surface. Such a regularity estimate is proved in  section~\ref{sec8}. Concluding remarks are given in section~\ref{sectconcl}.

\section{Problem formulation}\label{sec2}

Consider a surface $\Gamma(t)$ passively advected by a smooth velocity field $\bw=\bw(x,t)$, i.e. the normal velocity of $\Gamma(t)$ is given by $\bw \cdot \bn$, with
$\bn$ the unit normal on $\Gamma(t)$. We assume that for all $t \in [0,T] $,  $\Gamma(t)$ is a smooth hypersurface that is  closed ($\partial \Gamma =\emptyset$), connected, oriented, and contained in a fixed domain $\Omega \subset \Bbb{R}^d$, $d=2,3$. In the remainder we consider $d=3$, but all results have analogs for the case $d=2$.
The conservation of a scalar quantity
$u$ with a diffusive flux on $\Gamma(t)$ leads to the surface PDE (cf. \cite{James04}):
\begin{equation}
\dot{u} + ({\Div}_\Gamma\bw)u -{ \nu_d}\Delta_{\Gamma} u=0\qquad\text{on}~~\Gamma(t), ~~t\in (0,T],
\label{transport}
\end{equation}
 with initial condition $u(x,0)=u_0(x)$ for $x \in \Gamma_0:=\Gamma(0)$. Here $\dot{u}= \frac{\partial u}{\partial t} + \bw \cdot \nabla u$ denotes the advective material derivative,
 ${\Div}_\Gamma:=\operatorname{tr}\left( (I-\bn\bn^T)\nabla\right)$ is the surface divergence and $\Delta_\Gamma$ is the
 Laplace-Beltrami operator, $\nu_d>0$ is the constant diffusion coefficient.

In the analysis of partial differential equations it is  convenient to reformulate \eqref{transport} as a problem with
homogeneous initial conditions and a non-zero right-hand side.
To this end, consider the decomposition of the solution $u=\widetilde{u}+u^0$, where $u^0(\cdot,t):\ \Gamma(t) \to \Bbb{R}$, with $t\in[0,T]$, is chosen sufficiently smooth and such that
$
u^0(x,0)=u_0(x)$ on $\Gamma_0$, and $\frac{d}{dt}\int_{\Gamma(t)} u^0 \, ds =0.
$
Since the solution of \eqref{transport} has the  mass conservation property $\frac{d}{dt}\int_{\Gamma(t)} u\,ds=0$,
 the new unknown function $\widetilde{u}$ satisfies $\tilde u(\cdot,0)=0$ on $\Gamma_0$ and has
the zero mean property:
\begin{equation}\label{average}
\int_{\Gamma(t)} \widetilde{u}\,ds=0\quad\text{for all}~t\in [0,T].
\end{equation}
For this transformed function the surface diffusion equation takes the form
\begin{equation}
\begin{aligned}
\dot{\widetilde{u}} + ({\Div}_\Gamma\bw)\widetilde{u} -{\nu_d}\Delta_{\Gamma} \widetilde{u}&=f\qquad\text{on}~~\Gamma(t), ~~t\in (0,T],\\
\widetilde{u}(\cdot,0)&=0\qquad\text{on}~~\Gamma_0.
\end{aligned}
\label{transport1}
\end{equation}
The right-hand side is now non-zero: $f:=-\dot{u^0} - ({\Div}_\Gamma\bw)u^0 + {\nu_d}\Delta_{\Gamma}u^0$.
Using the Leibniz formula
\begin{equation}\label{Leibniz}
  \int_{\Gamma(t)} \dot v + v \DivG \bw \, ds = \frac{d}{dt}  \int_{\Gamma(t)} v \, ds, 
\end{equation}
and the integration by parts over $\Gamma(t)$, one immediately finds $\int_{\Gamma(t)} f\,ds=0$ for  all $t\in [0,T]$. In the remainder we consider the transformed problem \eqref{transport1} and write $u$ instead of $\tilde u$. 
In the stability analysis in section~\ref{sec4} we will use the zero mean property of $f$ and the corresponding zero mean property \eqref{average} of $u$.

\subsection{Weak formulation}
In this paper we present an error analysis of a finite element method for \eqref{transport1} and hence we need a suitable
weak formulation of this equation.  While several weak formulations of \eqref{transport1} are known in the literature, see \cite{Dziuk07,GReusken2011}, the most appropriate for our purposes is the
integral space-time formulation of \eqref{transport1} proposed in \cite{ORXsinum}. In this section we recall this formulation.
Consider  the space-time manifold
 \[
 \Gs= \bigcup\limits_{t \in (0,T)} \Gamma(t) \times \{t\},\quad  \Gs\subset \Bbb{R}^{4}.
 \]
Due to the identity
\begin{equation}\label{transform}
 \int_0^T \int_{\Gamma(t)} f(s,t) \, ds \, dt = \int_{\Gs} f(s) (1+ (\bw \cdot \bn)^2)^{-\frac12}\, ds,
\end{equation}
the scalar product $(v,w)_0=\int_0^T \int_{\Gamma(t)} v w \, ds \, dt$
induces a norm   that is equivalent to the standard norm on $L^2(\Gs)$. For our purposes,
it is more convenient to consider the $(\cdot,\cdot)_0$ inner product on  $L^2(\Gs)$.
Let $\nablaG$ denote the tangential gradient for $\Gamma(t)$ and introduce  the Hilbert space
\begin{equation}
H=\{\, v \in L^2(\Gs)~|~ \|\nablaG v\|_{L^2(\Gs)} <\infty \, \}, \quad (u,v)_H=(u,v)_0+ (\nablaG u, \nablaG v)_0. \label{inner}
  \end{equation}
We consider the material derivative $\dot{u}$ of $u \in H$ as a distribution on $\Gs$.
In \cite{ORXsinum} it is shown that $C_0^1(\Gs)$ is  dense in $H$. If $\dot{u}$ can be extended to a bounded linear functional on $H$, we write $\dot u \in H'$ { and $\langle \dot u, v \rangle = \dot u(v)$ for $v \in H$.} Define the space
\[
  W= \{ \, u\in H~|~\dot u \in H' \,\}, \quad \text{with}~~\|u\|_W^2 := \|u\|_H^2 +\|\dot u\|_{H'}^2.
\]
In \cite{ORXsinum}   properties of $H$ and $W$ are analyzed.
Both spaces are Hilbert spaces and smooth functions are  dense in $H$ and $W$. We shall recall  other useful results for elements of $H$ and $W$ at those places in this paper, where we need them.

Define
\[
\Wo:=\{\, v \in W~|~v(\cdot, 0)=0 \quad \text{on}~\Gamma_0\,\}.
\]
The space $\Wo$ is well-defined,
since functions from $W$ have well-defined traces in $L^2(\Gamma(t))$ for any $t\in[0,T]$. 
We introduce the symmetric bilinear form
\[
  a(u,v)= \nu_d (\nablaG u, \nablaG v)_0 + (\DivG \bw\, u,v)_0, \quad u, v \in H,
\]
which is  continuous on $H\times H$:
\begin{equation*}\label{eq:continuity}
a(u,v)\le(\nu_d+\alpha_{\infty}) \|u\|_H\|v\|_H,\quad\text{with}~\alpha_{\infty}:=\|\DivG \bw\|_{L^\infty(\Gs)}.
\end{equation*}
The weak space-time formulation of \eqref{transport1} reads: Find $u \in \Wo$ such that
\begin{equation} \label{weakformu}
 \la\dot u,v\ra +a (u,v) = (f,v)_0 \quad \text{for all}~~v \in H.
\end{equation}

\subsection{Well-posedness result and stability estimate}
Well-posedness of \eqref{weakformu} follows from the following lemma from \cite{ORXsinum}.
\begin{lemma}\label{la:infsup}
The following properties of the bilinear form $\la\dot u,v\ra + a(u,v)$ hold.
\begin{description}
\item{a)} Continuity:
\[
  | \la\dot u,v\ra + a(u,v)| \le (1+\nu_d+\alpha_{\infty})\|u\|_W \|v\|_H \quad \text{for all}~~u \in W,~v \in H.
\]
\item{b)} Inf-sup stability:
 \begin{equation}\label{infsup}
  \inf_{0\neq u \in \Wo}~\sup_{ 0\neq v \in \overset{\phantom{.}}{H}} \frac{\la\dot u,v\ra + a(u,v)}{\|u\|_W\|v\|_H} \geq c_s>0.
 \end{equation}
\item{c)} The kernel of the adjoint mapping is trivial:
\[ \big[~\la\dot u,v\ra + a(u,v)=0~~\text{for some}~v\in H~~\text{and all}~u \in \Wo~\big]\quad \Longrightarrow \quad v=0.
\]
\end{description}
\end{lemma}
\medskip
As a  consequence of  Lemma~\ref{la:infsup} one obtains:
\begin{theorem} \label{mainthm1}
For any $f\in L^2(\Gs)$, the problem \eqref{weakformu} has a unique solution $u\in \Wo$. This solution satisfies the a-priori estimate
\begin{equation} \label{stabestimate}
\|u\|_W \le c_s^{-1} \|f\|_0.
\end{equation}
\end{theorem}
Related to these stability results for the continuous problem we make some remarks that are relevant for the stability analysis of the discrete problem in Section~\ref{sec4}.
\begin{remark} \label{remstab} \rm
We remark that Lemma~\ref{la:infsup} and Theorem~\ref{mainthm1} have been proved for a slightly more general surface PDE than the surface diffusion problem \eqref{transport1}, namely
\begin{equation}
\begin{aligned}
\dot{u} + \alpha\, {u} -{\nu_d}\Delta_{\Gamma} {u}&=f\qquad\text{on}~~\Gamma(t), ~~t\in (0,T],\\
{u}&=0\qquad\text{on}~~\Gamma_0,
\end{aligned}
\label{transport2}
\end{equation}
with $\alpha\in L^\infty(\Gs)$ and a generic right-hand side $f\in H'$, not necessarily satisfying the zero integral condition.
The stability constant $c_s$ in the inf-sup condition \eqref{infsup} can be taken as
\[
c_s= \frac{\nu_d}{\sqrt{2}}(1+\nu_d+\alpha_{\infty})^{-2} e^{-2T(\nu_d+\tilde c)},\quad \tilde c =\|\alpha-\frac12\DivG \bw\|_{L^\infty(\Gs)}, ~\text{with}~\alpha_{\infty}:=\|\alpha\|_{L^\infty(\Gs)}.
\]
This stability constant deteriorates if $\nu_d \downarrow 0$ or $T \to \infty$.
\end{remark}
\medskip

\begin{remark} \label{remGronwall}
 \rm A stability result similar to \eqref{stabestimate}, in a somewhat weaker norm (without the $\|\dot u\|_{H'}$ term), can be derived using Gronwall's lemma, cf.~\cite{Dziuk07}. In \eqref{weakformu} we then take $v=u_{|[0,t]}$, with $t \in (0,T]$, and using the Leibniz formula we get
\[
  \frac12 \int_{\Gamma(t)} u^2 \, ds +  \nu_d \int_0^t \int_{\Gamma(\tau)} (\nabla_\Gamma u)^2 \, ds d\tau =  \int_0^t \int_{\Gamma(\tau)} f u \, ds d\tau - \frac12 \int_0^t \int_{\Gamma(\tau)} \DivG \bw \, u^2 \, ds d\tau.
\]
Using standard estimates we obtain for $h(t):=\frac12 \int_{\Gamma(t)} u^2 \, ds + \nu_d \int_0^t \int_{\Gamma(\tau)} (\nabla_\Gamma u)^2 \, ds d\tau$:
\begin{equation} \label{Gronwbound}
 h(t) \leq \frac12 \|f\|_0^2 +(1+ \| \DivG \bw\|_{L^\infty(\Gs)}) \int_0^t h(\tau)\, d\tau \quad \text{for all}~~t \in [0,T],
\end{equation}
and using Gronwall's lemma this yields a stability estimate.
\end{remark}
\medskip

\begin{remark} \label{remPD}
\rm  In general, for the problem \eqref{weakformu} a deterioration of the stability constant for $T \to \infty$, cf. Remark~\ref{remstab}, can not be avoided. This is seen from the simple example of a contracting sphere with a uniform initial concentration $u_0$. The solution then is of the form $u(x,t)=u_0 e^{\lambda t}$, with $\lambda >0$ depending on the rate of contraction. This possible exponential growth is closely related to the fact that if we represent \eqref{weakformu} as
\[
   \dot u + Au =f, \quad A: H \to H' \quad \text{given by} ~~\la A u,v \ra = (\DivG \bw u, v)_0 + \nu_d (\nablaG u, \nablaG v)_0,
\]
the symmetric operator $A$ is not necessarily positive semi-definite. The possible lack of positive semi-definitness is caused by $\DivG \bw$, which can be interpreted as local area change: From the Leibniz formula we obain $\int_{\gamma(t)}\DivG \bw (s,t) \, \,d s = \frac{d}{dt} \int_{\gamma(t)} 1 \, \,d s = \frac{d}{dt} |\gamma(t)|$, with $\gamma(t)$ a (small) connected subset of the surface $\Gamma(t)$. If the surface is not compressed anywhere (i.e., the local area is constant or increasing) then  $ \DivG \bw \geq 0$ holds and $A$ is positive semi-definite. In general, however, one has expansion and compression in different parts of the surface. Note that if $ \DivG \bw =0$, i.e., no local area change, we can still have an arbitrary strong convection of $\Gamma(t)$. For example, a constant velocity field $\bw(x,t)=\bw_0 \in \Bbb{R}^3$, with $\|\bw_0\| \gg 1$. In the stability analysis of the discrete problem in Section~\ref{sec4} we restrict to the case that $A$ is \emph{positive
  definite}, cf. the comments in Remark~\ref{remMotiv}. Clearly, the problem then has a nicer mathematical structure. In particular the solution does not have exponentially growing components. The restriction to positive definite $A$ still allows interesting cases with small local area changes (of arbitrary sign) and (very) strong convection of $\Gamma(t)$.
Even for very simple convection fields, e.g. $\bw$ constant, $A$ can not be postive definite on the space $\Wo$, the trial space used in \eqref{weakformu}. This is due to the fact that for $u(x,t)=u(t)$, i.e. $u$ is constant in $x$, we have $\nablaG u=0$. We  deal with this problem by restricting to a suitable \emph{subspace}, as explained below.
\end{remark}
\medskip

We outline a stability result from \cite{ORXsinum} for the case if $A$ is positive definite on a subspace. Functions $u \in H$  obey the Friedrichs inequality
\begin{equation}\label{Fr}
\int_{\Gamma(t)} |\nabla_\Gamma u|^2 \, ds \geq c_F(t) \int_{\Gamma(t)} ( u- \frac{1}{|\Gamma(t)|} \bar u)^2 \, ds\quad\text{for all}~t\in[0,T],
\end{equation}
with $c_F(t) >0$ and $\bar u(t):= \int_{\Gamma(t)} u(s,t)\, ds$.
 A smooth solution to problem \eqref{transport1} satisfies the zero average condition \eqref{average} and so we may look for a weak solution from the following
subspace of $\Wo$:
\begin{equation} \label{Wt}
\widetilde{W}:= \{\, u \in \Wo~|~ \bar u(t) =0 \quad \text{for all}~~t \in [0,T]\,\}.
\end{equation}
Obviously, elements of $\widetilde{W}$  satisfy the Friedrichs inequality with $\bar u=0$. Exploiting this,
one obtains the following result.

\begin{proposition}\label{Prop2}
Assume  $f$ satisfies  $\int_{\Gamma(t)} f\,ds=0$ for  almost all $t\in [0,T]$. Then the solution $u \in \Wo$ of \eqref{weakformu} belongs to $\widetilde{W}$.
Additionally assume that there exists a $c_0 >0$ such that
\begin{equation} \label{ass7}
   \DivG \bw (x,t) +\nu_d c_F(t) \geq c_0 \quad \text{for all}~~x \in \Gamma(t),~t \in [0,T]
\end{equation}
holds.  Then the inf-sup property \eqref{infsup} holds, with $\Wo$ replaced by the subspace $\widetilde{W}$ and $c_s= \frac{\min \{\nu_d, c_0\}}{2\sqrt{2} (1+\nu_d + \alpha_\infty)^2} $, where $\alpha_{\infty}:=\|\DivG \bw\|_{L^\infty(\Gs)}$.
\end{proposition}
\medskip

If the condition in  \eqref{ass7} is satisfied then $A$ is positive definite on the  subspace $\widetilde{W}$. Due to the positive-definitness  the stability constant $c_s$ is  independent of $T$.

\section{Finite element method} \label{sectderivation}\label{sec3}

Consider a  partitioning of the time interval:  $0=t_0 <t_1 < \ldots < t_N=T$, with a uniform time step $\Delta t = T/N$. The assumption of a uniform time step is made to simplify the presentation, but is not essential. A time interval is denoted by $I_n:=(t_{n-1},t_n]$.  The symbol $\Gs^n$ denotes the space-time interface corresponding to $I_n$, i.e.,  $\Gs^n:=\cup_{t \in I_n}\Gamma(t)\times\{t\}$, and $\Gs:= \cup_{1 \leq n \leq N}  \Gs^n $. We introduce the following subspaces $H_n:=\{\, v \in H~|~v=0  \quad \text{on}~~\Gs \setminus \Gs^n\, \}$ of $H$,
and define the spaces
\[
  W_n= \{ \, v\in H_n~|~\dot v \in H_n' \,\}, \quad \|v\|_{W_n}^2 = \|v\|_{H}^2 +\|\dot v\|_{H_n'}^2.
\]
An element $(v_1, \ldots, v_N) \in \oplus_{n=1}^N W_n$ is identified with $v \in H$, by $v_{|\Gs^n}= v_n$.
Our finite element method is conforming with respect to the broken  trial space
\[
W^b:= \oplus_{n=1}^N W_n,~~\text{with norm}~~ \|v\|_{W^b}^2= \sum_{n=1}^N \|v_n\|_{W_n}^2= \|v\|_H^2+ \sum_{n=1}^N \|\dot v_n\|_{H_n'}^2 .
\]
For $u \in W_n$, the one-sided  limits
$u_+^{n}=u_+(\cdot,t_{n})$ and ${u}_{-}^n=u_{-}(\cdot,t_n)$ are well-defined in $L^2(\Gamma(t_n))$ (cf. \cite{ORXsinum}).  At $t_0$ and $t_N$  only $u_+^{0}$ and $u_{-}^{N}$ are defined.
For $v \in W^b$, a jump operator  is defined by $[v]^n= v_+^n-v_{-}^n \in L^2(\Gamma(t_n))$, $n=1,\dots,N-1$. For $n=0$, we define $[v]^0=v_+^0$.

On the cross sections $\Gamma(t_n)$, $0 \leq n \leq N$,  of $\Gs$ the $L^2$ scalar product is denoted by
$
 (\psi,\phi)_{t_n}:= \int_{\Gamma(t_n)} \psi \phi \, ds .
$
In addition to $a(\cdot,\cdot)$,  we define on the broken space $W^b$ the following bilinear forms:
\begin{align*}
 d(u,v)  =  \sum_{n=1}^N d^n(u,v), \quad d^n(u,v)=([u]^{n-1},v_+^{n-1})_{t_{n-1}},\quad
 \la \dot u ,v\ra_b =\sum_{n=1}^N  \la \dot u_n, v_n\ra.
\end{align*}

It is easy to check, see \cite{ORXsinum}, that the solution to \eqref{weakformu} also solves the following variational problem in the broken space: Find $u \in W^b$ such that
\begin{equation} \label{brokenweakformu}
  \la \dot u ,v\ra_b +a(u,v)+d(u,v) =( f,v )_0 \quad \text{for all}~~v \in W^b.
\end{equation}
This variational formulation uses $W^b$ as test space, since the term $d(u,v)$ is not well-defined for an arbitrary $v\in H$.
Also note that the initial condition $u(\cdot,0)=0$ is not an essential condition in the space $W^b$, but is treated in a weak sense (as is standard in DG methods for time dependent problems).
 From an algorithmic point of view,  this formulation has the advantage that due to the use of the broken space $W^b= \oplus_{n=1}^N W_n$ it can be solved in a time stepping manner. The discretization that we introduce below is a Galerkin method for the weak formulation \eqref{brokenweakformu}, with  a finite element space $W_{h}\subset W^b$.

To define this $W_{h}$, consider  the partitioning of the space-time volume domain $Q= \Omega \times (0,T] \subset \Bbb{R}^{3+1}$ into time slabs  $Q_n:= \Omega \times I_n$. Corresponding to each time interval $I_n:=(t_{n-1},t_n]$ we assume a given shape regular tetrahedral triangulation $\T_n$ of the spatial domain $\Omega$. The corresponding spatial mesh size parameter is denoted by $h$. 
Then $\mathcal{Q}_h=\bigcup\limits_{n=1,\dots,N}\T_n\times I_n$ is a subdivision of $Q$ into space-time
prismatic nonintersecting elements. We shall call $\mathcal{Q}_h$ a space-time triangulation of $Q$.
Note that this triangulation is not necessarily fitted to the surface $\Gs$. We allow $\T_n$ to vary with $n$ (in practice, during time integration one may wish to adapt the space triangulation depending on the changing local geometric properties of the surface) and so
the elements of $\mathcal{Q}_h$ may not  match at $t=t_n$.

The \emph{local} space-time triangulation $\mathcal{Q}_h^\Gs$ consists of space-time prisms that are intersected by $\Gs$, i.e., $\mathcal{Q}_h^\Gs = \{\,  T \times I_n \in \mathcal{Q}_h~|~{\rm meas}_3( (T \times I_n) \cap \Gs >0\,\}$, cf. Fig.~\ref{Fig1}. If $(T \times I_n) \cap \Gs$ consists of a face $F$ of the prism $T \times I_n$, we include in $\mathcal{Q}_h^\Gs$ only one of the two prisms that have this $F$ as their intersection. The (local) domain formed by all prisms in  $\mathcal{Q}_h^\Gs$ is denoted by $Q^\Gs$.

For any $n\in\{1,\dots,N\}$, let $V_n$ be the finite element space of continuous piecewise affine functions on $\T_n$.
We define the (local) \emph{volume space-time finite element space}:
\[
V_{h}:= \{ \, v: Q^\Gs  \to \Bbb{R} ~|~  v(x,t)= \phi_0(x) + t \phi_1(x)~\text{on every}~Q_n\cap Q^\Gs,~\text{with}~\phi_0,\, \phi_1 \in V_n\,\}.
\]
Thus, $V_{h}$ is a space of piecewise bilinear functions with respect to  $\mathcal{Q}_h^\Gs$, continuous in space and discontinuous in time. Now we define our \emph{surface finite element space} as the space
of traces of functions from $V_{h}$ on $\Gs$:
\begin{equation} \label{deftraceFE}
  W_{h} := \{ \, w:\Gs \to \Bbb{R}~|~ w=v_{|\Gs}, ~~v \in V_{h} \, \}.
\end{equation}

\begin{figure}[ht!]
 \centering
    \includegraphics[width=0.8\textwidth]{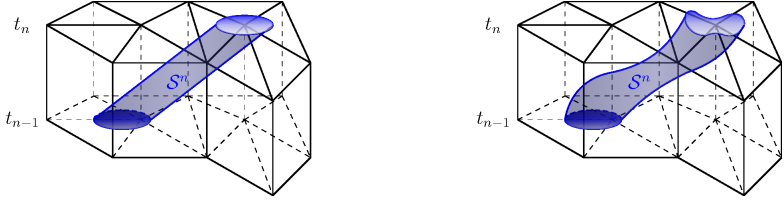}
\caption{Illustration of the local space-time triangulation $\mathcal{Q}_h^\Gs$ in one time slab. In the left picture we have a constant $\bw$, hence \eqref{ass7} is satisfied.}
 \label{Fig1}
\end{figure}

The finite element method reads: Find $u_h \in W_{h}$ such that
\begin{equation} \label{brokenweakformu_h}
  \la \dot u_h ,v_h\ra_b +a(u_h,v_h)+d(u_h,v_h) = (f,v_h)_0 \quad \text{for all}~~v_h \in W_{h}.
\end{equation}
As usual in time-DG methods,  the initial condition for $u_h(\cdot,0)$ is treated in a weak sense.
Due to $u_h\in H^1(Q_n)$ for $n=1,\dots,N$, the first term  in \eqref{brokenweakformu_h}
can be written as
\begin{equation} \label{eval}
\la \dot u_h ,v_h\ra_b=\sum_{n=1}^N\int_{t_{n-1}}^{t_n}\int_{\Gamma(t)} (\frac{\partial u_h}{\partial t} +\bw\cdot\nabla u_h)v_h ds\,dt.
\end{equation}
In the (very unlikely) case that $\Gamma(t)$ is a face of two tetrahedra $T_1$,$T_2$ and both $T_1 \times I_n$ and $T_2 \times I_n$ are contained in $\mathcal{Q}_h^\Gs$, we use a simple averaging in the evaluation of $\bw\cdot\nabla u_h$ in  \eqref{eval}.
Recall that the solution of  the continuous problem \eqref{transport1} satisfies the zero mean condition \eqref{average}, which
corresponds to the mass conservation law valid for the original problem \eqref{transport}.
We investigate whether the  condition \eqref{average} is preserved for the finite element formulation \eqref{brokenweakformu_h}.

Assume that $u_h$ is a solution of \eqref{brokenweakformu_h}. Denote $\bar u_h(t)=\int_{\Gamma(t)} u_h ds$.
We have  $\int_{\Gamma(t)} f\,ds=0$ for all $t>0$. In \eqref{brokenweakformu_h}, set $v_h=1$ for $t\le t_n$ and $v_h=0$ for $t > t_n$. This implies $
\bar u_{h,-}(t_{n}):=\int_{\Gamma(t_n)} u_{-}^n\,ds=0$ for $n=0,1,\dots$.
 Setting $v_h=t-t_{n-1}$ for $t_{n-1}\le t\le t_n$ and $v_h=0$ otherwise,
we additionally get $\int_{t_{n-1}}^{t_n} \bar u_h (t)\, dt=0$. Summarizing, we obtain the following:
\begin{equation}\label{means}
\bar u_{h,-}(t_{n})=0\quad\text{and}\quad\int_{t_{n-1}}^{t_n} \bar u_h(t)\, dt=0,\quad n=1,2,\dots.
\end{equation}
For a \emph{stationary} surface, $\bar u_h(t)$ is a piecewise affine function and  thus \eqref{means} implies $\bar u_h(t)\equiv0$, i.e., we have exact mass conservation on the discrete level.
 If the surface evolves, the finite element method is not necessarily mass conserving:  \eqref{means} holds, but $\bar u_h(t) \neq 0$ may occur for $ t_{n-1} \le t < t_n$. To enforce a better mass conservation and
 enhance stability  of the finite element method, cf. Remark~\ref{remMotiv},  we introduce a \emph{consistent} stabilizing term involving the quantity $\bar u_h (t)$ to the discrete bilinear form. More precisely, define
\begin{equation} \label{stabblf}
  a_\sigma(u,v):= a(u,v)+\sigma \int_0^T \bar u(t) \bar v(t) \, dt, \quad \sigma \geq 0.
\end{equation}
Instead of \eqref{brokenweakformu_h} we consider the stabilized version: Find $u_h \in W_{h}$ such that
\begin{equation} \label{brokenweakformu_h1}
  \la \dot u_h ,v_h\ra_b +a_\sigma(u_h,v_h)+d(u_h,v_h) =(f,v_h)_0 \quad \text{for all}~~v_h \in W_{h}.
\end{equation}
As mentioned above, taking $\sigma >0$ we expect both a stabilizing effect and an improved mass conservation property. Adding this stabilization term does not lead to significant additional computational costs for computing the stiffness matrix, cf. Section~\ref{sectimplem}.

For the solution $u\in W$ of \eqref{brokenweakformu}, the stabilization term vanishes: $\bar u(t)=0$.
Therefore the error $e=u-u_h$ of the finite element method \eqref{brokenweakformu_h1} satisfies the
Galerkin orthogonality relation:
\begin{equation}\label{Galerkin}
 \la \dot e ,v_h\ra_b +a_\sigma(e,v_h)+d(e,v_h) =0 \quad \text{for all}~~v_h \in W_{h}.
\end{equation}

\subsection{Implementation aspects} \label{sectimplem}
We comment on a few implementation aspects. More details are found in the recent  article \cite{refJoerg}.

By choosing the test functions $v_h$ in \eqref{brokenweakformu_h1} per time slab, as in standard space-time DG methods, one obtains an implicit time stepping algorithm.
 Two main implementation issues are the approximation of the space-time integrals in the bilinear form $\la \dot u_h ,v_h\ra_b +a_\sigma(u_h,v_h)$ and the representation of the  finite element trace functions in $W_{h}$.  To approximate the integrals, one makes use of the transformation formula \eqref{transform} converting space-time integrals to surface integrals over $\Gs$, and next one
approximates $\Gs$ by a `discrete' surface $\Gs^h$; this is done locally, i.e. time slab per time slab.    The approximate surface $\Gs^h$ can be the zero level of $\phi_h\in W_{\hat h}$,
where $\phi_h$ is a bilinear finite element approximation of a level set function $\phi(x,t)$, the zero level of which is the surface $\Gs$.  To reduce the ``geometric error'' it may be efficient  to determine $\phi_h \in W_{\hat h}$  in a finite element space  with mesh size $\hat h < h$, $\hat{\Delta t} < \Delta t$, e.g., $\hat h = \frac 12 h$, $\hat{\Delta t} = \frac12 \Delta t$ (one regular refinement of the given
outer space-time mesh).
Within each space-time prism $T\times I_n \in \mathcal{Q}_h^\Gs$ the zero level of  $\phi_h\in W_{\hat h}$ can be represented as a union of tetrahedra, cf. \cite{refJoerg}, and standard quadrature formulas can be used.
 Results of numerical experiments obtained using  such treatment of integrals over $\Gs$ are reported
in \cite{refJoerg,GOReccomas,ORXsinum}.

For the representation of  the finite element  functions in $W_{h}$ it is natural to use traces of the standard nodal basis functions in the volume space-time finite element space $V_{h}$.   In general, these trace functions form a frame in $W_{h}$. A finite element surface solution is represented as a linear combination of the elements from this frame.
Linear systems resulting in every time step  may have more than one solution, but every solution yields the same trace function, which is the unique solution of \eqref{brokenweakformu_h1}. If $\Delta t \sim h$ and $\|\bw \|_{L^\infty(\Gs)} = \mathcal{O}(1)$ then the number of tetrahedra $T \in \T_n$ that are intersected by $\Gamma(t)$, $t \in I_n$, is of the order $\mathcal{O}(h^{-2})$. Hence, per time step the linear systems have $\mathcal{O}(h^{-2})$ unknowns, which is the same complexity as a discretized spatially \emph{two}-dimensional elliptic problem. Note that although we derived the method in $\Bbb{R}^{3+1}$, due to the time stepping and the trace operation, the discrete problems have two-dimensional complexity. Since the discrete problems have a complexity of (only) $\mathcal{O}(h^{-2})$ it may be efficient to use  a  sparse direct solver for computing the discrete solution. Linear algebra aspects of the surface finite element method  have been addressed in \cite{OlshanskiiReusken08} and will be
further investigated in future work.

 The stabilization term   in \eqref{stabblf} does not cause significant additional computational work. In one time slab it has the form $\int_{t_{n-1}}^{t_n} \bar u(t) \bar v(t) \, dt$.  Let $\phi_i$, $1 \leq i \leq M$,  denote the nodal basis functions in the outer space $V_h$, then the $M \times M$- matrix representing this bilinear form has entries $\int_{t_{n-1}}^{t_n}  \int_{\Gamma(t)} \phi_j \, ds  \int_{\Gamma(t)} \phi_i \, ds\, dt$. If quadrature for $\int_{t_{n-1}}^{t_n}$, with nodes $\xi_1,\ldots, \xi_k \in [t_{n-1},t_n]$, is applied this results in a stabilization  matrix of the form $S=\sum_{r=1}^k \alpha_r z_r z_r^T$, with $\alpha_r \in\Bbb{R}$, $z_r \in \Bbb{R}^M$. The vector $z_r$ has entries $(z_r)_i= \int_{\Gamma(\xi_r)} \phi_i(s, \xi_r)\, ds$. We need only a few quadrature points, e.g. $k=2$, hence $S$ is  a sum of only a few rank one matrices. Since the stabilization matrix is symmetric positive semi-definite it also  improves the conditioning of the stiffness matr
 ix.

\section{Stability of the finite element method} \label{sec4}
We present a stability analysis of the discrete problem \eqref{brokenweakformu_h1} for the positive definite case, cf. Remark~\ref{remPD}. In Remark~\ref{remMotiv} below we explain why we restrict ourselves to the positive definite case and comment on the role of the stabilization. We introduce the following mesh-dependent norm:
\[
\enorm{u}_h:=\left(\max_{n=1,\dots,N}\|u_{-}^n\|_{t^n}^2 + \sum_{n=1}^N \|[u]^{n-1}\|_{t_{n-1}}^2+\|u\|_H^2\right)^\frac12.
\]

\begin{theorem} \label{stabadd}
Assume \eqref{ass7}  and take $ \sigma \geq  \frac{\nu_d}2\max\limits_{t\in[0,T]}\frac{c_F(t)}{|\Gamma(t)|}$, where $c_F(t)$ is defined in \eqref{Fr}. Then
the inf-sup estimate
\begin{equation} \label{infsuph}
 \inf_{u\in W^b} \sup_{v\in W^b}\frac{\la \dot u,v\ra_b +a_\sigma(u,v)+d(u,v)}{\enorm{v}_h \enorm{u}_h} \geq c_s
\end{equation}
and the ellipticity estimate
\begin{equation} \label{coer}
 \la \dot u,u\ra_b +a_\sigma(u,u)+d(u,u) \geq 2c_s\left(\|u_{-}^N\|_T^2 + \sum_{n=1}^N \|[u]^{n-1}\|_{t_{n-1}}^2+\|u\|_H^2\right)
\end{equation}
for all $ u \in W^b$ hold,  with $c_s=\frac14\min\{1,\nu_d,c_0\}$ and $c_0$ from \eqref{ass7}. The results in \eqref{infsuph}, \eqref{coer} also hold
with $W^b$  replaced by $W_h$.
\end{theorem}
\begin{proof} Take $u \in W^b$, $u \neq 0$, and
let $M\in\{1,\dots,N\}$. Set $\tilde{u}=u$ for $t\in(0,t_M]$ and  $\tilde{u}=0$ for $t\in(t_M,T)$.
Applying  partial integration  on every time interval we get
\[
  \la \dot u, \tilde{u} \ra_b= \frac12 \sum_{n=1}^M \Big( \|u_{-}^n\|_{t_n}^2 - \|u_{+}^{n-1}\|_{t_{n-1}}^2 \Big)
  -\frac12 \int_0^{t_M}(\DivG \bw, u^2)_{\Gamma(t)}\,dt .
\]
It is also straightforward to derive
\[
d(u,\tilde{u})=  - \frac12 \sum_{n=1}^M \Big( \|u_{-}^n\|_{t_n}^2 -  \|u_{+}^{n-1}\|_{t_{n-1}}^2 \Big)   + \frac12  \|u_{-}^M\|_{t_M}^2 + \frac12 \sum_{n=1}^M \|[u]^{n-1}\|_{t_{n-1}}^2.
\]
The Friedrichs inequality \eqref{Fr} yields
\[
  \int_{\Gamma(t)} |\nabla_\Gamma u|^2 \, ds \geq  c_F(t) \big( \int_{\Gamma(t)} u^2 \, ds - \frac{1}{|\Gamma(t)|} \bar u^2(t) \big).
\]
Using this, we get
\begin{align*}
  & a_\sigma(u,\tilde{u})   =  \int_0^{t_M} \nu_d \|\nablaG u\|^2_{L^2(\Gamma(t))} + (\DivG \bw, u^2)_{L^2(\Gamma(t))}  + \sigma \bar u(t)^2 \, dt\\
 & \geq \int_0^{t_M} \frac12   (\nu_d c_F + 2\DivG \bw ,u^2)_{L^2(\Gamma(t))} + (\sigma-\frac{\nu_d}2\frac{c_F(t)}{|\Gamma(t)|})  \bar u(t)^2  +  \frac{\nu_d}2\|\nablaG u\|^2_{L^2(\Gamma(t))}\,dt \\ &  \geq    \int_0^{t_M} \frac12 (\nu_d c_F + 2\DivG \bw ,u^2)_{L^2(\Gamma(t))}+ \frac{\nu_d}{2}\|\nablaG u\|^2_{L^2(\Gamma(t))}\,dt.
\end{align*}
Combining the relations above and using \eqref{ass7}, we get
\begin{multline}
\label{aux509}
  \la \dot u ,\tilde{u}\ra_b +a_\sigma(u,\tilde{u})+d(u,\tilde{u})  \geq
 \frac12\left(\|u_{-}^M\|_{t_M}^2 + \sum_{n=1}^M \|[u]^{n-1}\|_{t_{n-1}}^2\right.\\\left.+ \int_0^{t_M}c_0\|u\|^2_{L^2(\Gamma(t))}+\nu_d\|\nablaG u\|^2_{L^2(\Gamma(t))}\,dt\right).
\end{multline}
Taking $M=N$ in this inequality proves \eqref{coer}.
Let $M$ be such   that $\|u_{-}^M\|_{t^M}=\max_{n=1,\dots,N}\|u_{-}^n\|_{t^n}^2$.
Setting $v=\tilde{u}+u$, using \eqref{aux509} and performing obvious computations gives \eqref{infsuph}.
Since $ W_{h}\subset W^b$ and $u\in W_h\,\Rightarrow\,\tilde{u}\in W_{h}$, the results in \eqref{infsuph}, \eqref{coer} also hold on the finite element subspace.\quad
\end{proof}
\medskip
\noindent In this stability result there are no restrictions on the size of  $h$ and $\Delta t$. In particular the stability is guaranteed even if $\Delta t$ is large. This is in agreement with the \emph{strong robustness of the method},  observed in the numerical experiments in  \cite{refJoerg,ORXsinum,GOReccomas}.\medskip

\begin{remark} \label{remMotiv} \rm We comment on the assumptions we use in Theorem~\ref{stabadd}.  An inf-sup result in $W^b$, similar to \eqref{infsuph}, can also be derived for the general (indefinite) case, i.e., without assuming \eqref{ass7}, and without stabilization. Such a result is given in Lemma 5.2 in \cite{ORXsinum}. The proof uses a test function of the form $v= \mu e^{-\gamma t} u +  z$, with a suitable $\mu>0,\,\gamma >0$ and $z \in W^b$. The factor $e^{-\gamma t}$ is used to control the term $(\DivG \bw u, u)_0$. Of course, the stability constant then depends on $T$ and deteriorates for $T \to \infty$. For the discrete space $W_h$, however, we are not able to derive a stability result for the general (indefinite) case. The key point is that for $u_h \in W_h$ a test function of the form $e^{-\gamma t} u_h$ is not allowed, since it is not an element of the test space $W_h$. Using an approximation (interpolation or projection) of $e^{-\gamma t} u_h$ in the finite element
  space we
  are not able to get
sufficient control of the term  $(\DivG \bw u, u)_0$. A similar difficulty, for the general problem, arises if one applies a discrete analogon of the Gronwall argument outlined in Remark~\ref{remGronwall}: Let $u = u_h \in W_h$ be a finite element function. For the corresponding test function one can take  $v=\tilde{u}$ as in the proof above, i.e., $v=u_{|[0,t_M]}$. Taking $\sigma =0$ we obtain
\[ \begin{split}
  & \frac12 \|u_{-}^M\|_{t_M}^2 +\frac12 \sum_{n=1}^M \|[u]^{n-1}\|_{t_{n-1}}^2+  \nu_d\int_0^{t_M} \int_{\Gamma(t)} (\nablaG u)^2 \, ds \,dt \\ & = \int_0^{t_M} \int_{\Gamma(t)} fu \,ds \,dt - \frac12  \int_0^{t_M} \int_{\Gamma(t)} \DivG \bw u^2 \, ds \, dt.
\end{split} \]
Define $h(t):=  \frac12 \int_{\Gamma(t)} u^2 \, ds  + \sum_{n=1}^M \|[u]^{n-1}\|_{t_{n-1}}^2+ \nu_d \int_0^t \int_{\Gamma(\tau)} (\nabla_\Gamma u)^2 \, ds d\tau$, for $t \in (t_{M-1},t_M]$, $M=1, \ldots N$. With similar arguments as in  Remark~\ref{remGronwall} we get  the estimate
\[
 h(t_M) \leq \frac12 \|f\|_0^2 + (1+\|\DivG \bw\|_{L^\infty(\Gs)}) \int_{0}^{t_M} h(\tau)\, d\tau, \quad M=1, \ldots,N,
\]
cf. \eqref{Gronwbound}. Define $ a_M=h(t_M)$. For a discrete Gronwall lemma we need an inequality of the form $a_M \leq c+ \sum_{1 \leq k < M} g_k a_k$, $M=1,\ldots ,N$.
In our case we have to control $\int_{0}^{t_M} h(\tau)\,d\tau$ by the values $h(t_k)$, $k=0,\ldots,M$. For a stationary $\Gamma(t)$, this can be realized using the fact that $u$ is linear w.r.t. $t$ on $I_n$. For an evolving $\Gamma(t)$, however, the function $h(t)$ can have rather general behavior and it is not clear under which reasonable assumptions the integral can be bounded by the function values $h(t_k)$.

\emph{In view of these observations we restrict analysis to the nicer  positive definite case,  hence we assume that \eqref{ass7} holds}. As mentioned
in Remark~\ref{remPD}, condition \eqref{ass7} is not sufficient for $A$ to be positive definite on $W_h$. The difficulty comes from the functions  $u(x,t)$ that are constant in spatial directions. For the continuous case we dealt with this problem by restricting to the subspace $\widetilde{W}$, cf.~\eqref{Wt}. In case of an evolving $\Gamma(t)$, requiring the discrete solution $u_h$ to lie in  $\widetilde{W}$ is a too strong condition, which leads to an unacceptable reduction of  the degrees of freedom (often, only $u_h=0$ is allowed). This is the reason, why we introduce the stabilization. For $\sigma$ sufficiently large the corresponding stabilized operator $A_\sigma$ is positive definite on $W_h$. In numerical experiments we observe that in general $\sigma =0$ results in a stable method. We have  the following heuristic explanation for this. The discrete solution remains the same if we restrict the discretization to the subspace $\widetilde{W}_h \subset W_h$ of functions that sati
 sfy \eqref{means}. The distance of this space $\widetilde{W}_h$  to  $\widetilde{W}^b=\{\, u \in W^b~|~ \bar u(t)=0\,\}$ is expected to be small. On the latter space the operator $A$ without stabilization  is positive definite (if \eqref{ass7} holds) and thus it is plausible that this positive definiteness holds on $\widetilde{W}_h$, too.
\end{remark}
\ \\

The ellipticity result \eqref{coer} is sufficient for existence of a unique solution and \eqref{infsuph} yields a priori bound in $\enorm{\cdot}_h$ norm. We summarize this in  the following proposition.
\begin{proposition}
Assume \eqref{ass7}  and take $ \sigma$ as in Theorem~\ref{stabadd}. Then the discrete problem \eqref{brokenweakformu_h1}
has a unique solution $u_h \in W_{h}$. For $u_h$  the a priori estimate
\begin{equation} \label{aprior}
\enorm{u_h}_h \le c_s^{-1}  \|f\|_{0}.
\end{equation}
holds, with $c_s$ as in Theorem~\ref{stabadd}.
\end{proposition}

\section{Continuity result} \label{sec5}
We derive continuity results for the bilinear form of the finite element method.
\begin{lemma} \label{Lem_cont}
For any $e,v\in W^b$  the following  holds, with constants $c$ independent of $e, v, h, N$:
\begin{align} \label{cont}
  |\la \dot e,v\ra_b +a_\sigma(e,v)+d(e,v)|&
\leq c  \enorm{v}_h(\|e\|_{W^b}+\sum_{n=0}^{N-1}\|[e]^n\|_{t^n}),\\
|\la \dot e,v\ra_b +a_\sigma(e,v)+d(e,v)|&
\leq c  \enorm{e}_h(\|v\|_{W^b}+\sum_{n=1}^{N-1}\|[v]^n\|_{t^n}+\|v\|_T). \label{cont_dual}
\end{align}
\end{lemma}
\begin{proof}
The stabilizing term in $a_\sigma(e,v)$ is estimated as follows:
\begin{equation}\label{est_stab}
\begin{aligned}
\left|\sigma \int_0^T \int_{\Gamma(t)} e \, dx \int_{\Gamma(t)} v dx \, dt\right|&\le
\sigma \int_0^T |\Gamma(t)| \left(\int_{\Gamma(t)} e^2 dx\right)^{\frac12}  \left(\int_{\Gamma(t)} v^2 dx\right)^{\frac12} \, dt\\ &\le  \sigma \max\limits_{t\in[0,T]}|\Gamma(t)|\|e\|_0\|v\|_0.
\end{aligned}
\end{equation}
The material derivative term is treated using integration by part:
\begin{equation*}
\begin{aligned}
  \la \dot e,\right.&\left.v \ra_b= \sum_{n=1}^N \Big( (e_{-}^n,v_{-}^n)_{t_n} - (e_{+}^{n-1},v_{+}^{n-1})_{t_{n-1}}\Big)
  - (\DivG \bw\, e, v )_0 - \la \dot v,e \ra_b\\
  &= -\sum_{n=1}^N([e]^{n-1},v_{+}^{n-1})_{t_{n-1}} -\sum_{n=1}^{N-1}([v]^{n},e_{-}^{n})_{t_{n}}+ (e_{-}^{N},v)_{T}
  - (\DivG \bw\, e, v )_0 - \la \dot v,e \ra_b\\
   &= -d(e,v) -\sum_{n=1}^{N-1}([v]^{n},e_{-}^{n})_{t_{n}} + (e_{-}^{N},v)_{T}
  - (\DivG \bw\, e, v )_0 - \la \dot v,e \ra_b.
\end{aligned}
\end{equation*}
Now we use the relation $\la \dot v,e \ra_b=\sum_{n=1}^N\la \dot v_n,e_n \ra\  $  and
the Cauchy inequality to estimate
\begin{equation}\label{est_deriv}
 \begin{split}
  |\la \dot e,v \ra_b+d(e,v)| & \le
  \|e_{-}^{N}\|_T\|v\|_{T} +\alpha_{\infty} \|e\|_0 \|v\|_0 + \|e\|_H\left(\sum_{n=1}^N\|\dot v_n\|^2_{H'_n}\right)^{\frac12}
  \\ & +\max_{n=1,\dots,N-1}\|e_{-}^n\|_{t_n}\sum_{n=1}^{N-1}\|[v]^{n}\|_{t_{n}} .
\end{split}
\end{equation}
Combining \eqref{est_stab}, \eqref{est_deriv}, and $a(e,v)\le  \nu_d \|\nablaG e\|_0\|\nablaG v\|_0 + \alpha_{\infty}\|e\|_0\|v\|_0$,  we get
\begin{align*}
  & |\la \dot e,v\ra_b +a_\sigma(e,v)+d(e,v)| \\
 & \le
 \|e_{-}^{N}\|_T\|v\|_{T} +(2\alpha_{\infty}+\sigma\max\limits_{t\in[0,T]}|\Gamma(t)|)\|e\|_0 \|v\|_0 + \|e\|_H\left(\sum_{n=1}^N\|\dot v_n\|^2_{H'_n}\right)^{\frac12}\\  & + \nu_d \|\nablaG e\|_0\|\nablaG v\|_0 +\max_{n=1,\dots,N-1}\|e_{-}^n\|_{t_n}\sum_{n=1}^{N-1}\|[v]^{n}\|_{t_{n}}.
\end{align*}
The Cauchy inequality and the definition of the norms $\enorm{e}_h$, $\|v\|_{W^b}$ imply the result in \eqref{cont_dual}.
The inequality in \eqref{cont} is proved by the same arguments, but skipping the integration by parts step.
\quad\end{proof}
\medskip

The norm $\enorm{\cdot}_h$ is weaker than the norm $\|\cdot\|_W$ used for the stability analysis of the original
`differential' weak formulation \eqref{weakformu}, since the latter norm provides control over the material derivative in  $H'$. For the discrete solution we can establish control over the
material derivative  only in a weaker sense, namely in a space dual to the discrete space. Indeed, using estimates as in
the proof of Lemma~\ref{Lem_cont} we get
\[
|a_\sigma(u_h,v)|\\ \leq \enorm{u_h}_h \left((\alpha_{\infty}+\sigma\max\limits_{t\in[0,T]}|\Gamma(t)|)^2\|v\|_0^2+\nu_d^2\|\nablaG v\|_0^2\right)^{\frac12}\le c\,\enorm{u_h}_h \|v\|_H,
\]
and thus for the discrete solution $u_h \in W_{h}$ of \eqref{brokenweakformu_h1} one obtains,
using \eqref{aprior}:
\begin{equation} \label{aprior2}
\sup_{v \in W_{h}}\frac{\la \dot u_h,v \ra_b + d(u_h,v)}{\|v\|_H} = \sup_{v \in W_{h}}\frac{(f,v_h)_0 - a_\sigma(u_h,v)}{\|v\|_H} \le   c  \|f\|_{0}.
\end{equation}

\section{Discretization error analysis} \label{secterroranalysis} \label{sec6}
In this section we prove an error bound for the discrete problem \eqref{brokenweakformu_h1}. The analysis is based on the usual arguments, namely the stability estimate derived above combined with the Galerkin orthogonality and interpolation error bounds. The  surface finite element space is the trace of an outer volume finite element space $V_h$. For the analysis of the discretization error in the surface finite element space we use information on the approximation quality of the outer space. Hence, we need a suitable extension procedure for smooth functions on the space-time manifold $\Gs$. This topic is addressed in subsection~\ref{sec_ext}.

\subsection{Extension of functions defined on  $\Gs$}\label{sec_ext}
For a function $u \in H^2(\Gs)$ we need an extension $u^e \in H^2(U)$, where $U$ is a neighborhood in $\Bbb{R}^4$ that contains the space-time manifold $\Gs$.  Below we introduce  such an extension 
and derive some properties that we need in the analysis. We extend $u$ in a \emph{spatial} normal direction to $\Gamma(t)$ for every $t\in [0,T]$. For this procedure to be well-defined and the properties to hold, we need sufficient smoothness of the manifold $\Gs$. We assume $\Gs$ to be a three-dimensional $C^3$-manifold in $\Bbb{R}^4$.

For  some $\delta>0$ let
\begin{equation} \label{defU}
 U = \{\, \bx:=(x,t) \in  \R^{3+1}~|~{\rm dist}(x,\Gamma(t)) < \delta\, \}
\end{equation}
be a neighborhood of $\Gs$. The value of $\delta$ depends on  curvatures  of $\Gs$ and will be specified below. Let $d: U \rightarrow \R$ be the
signed distance function, $|d(x,t)|:={\rm dist}(x,\Gamma(t))$ for all
$\bx \in U$. Thus, $\Gs$ is the zero level set of $d$.  The spatial gradient  $\bn_{\Gamma}
=\nabla_x d \in \Bbb{R}^3$ is the exterior normal vector for $\Gamma(t)$. The normal vector for $\Gs$ is
\[
\bn_{\Gs}=\nabla d/\|\nabla d\|= \frac{1}{\sqrt{1+V_\Gamma^2}}(\bn_{\Gamma}, - V_\Gamma)^T \in \Bbb{R}^4,\quad V_\Gamma=\bw\cdot \bn_\Gamma.
\]
Recall that $V_\Gamma$ is the normal velocity of the evolving surface $\Gamma(t)$.
The normal $\bn_{\Gamma}$ has a natural extension given by  $\bn(\bx):=\nabla_x d(\bx) \in \Bbb{R}^3$ for all $\bx \in U$. Thus,
$\bn=\bn_{\Gamma}$ on $\Gs$ and $\|\bn (\bx)\|=1$ for all $\bx\in U$.
The spatial Hessian of $d$ is denoted by $\bH \in \Bbb{R}^{3\times3}$. The eigenvalues of $\bH$ are  $\kappa_1(x,t),
\kappa_2(x,t)$, and 0. For  $x \in \Gamma(t)$ the eigenvalues $\kappa_i(x,t)$, $i=1,2$, are
the principal curvatures of $\Gamma(t)$. Due to the smoothness assumptions on $\Gs$,
the principal curvatures are uniformly bounded in space and time:
\[
\sup_{t\in [0,T]}\sup_{x\in\Gamma(t)}(|\kappa_1(x,t)|+|\kappa_2(x,t)|)\le \kappa_{\max}.
\]
We introduce a local coordinate system by using the  projection $\bp:\, U \rightarrow
\Gs $:
\[
 \bp(\bx)=\bx-d(\bx)(\bn(\bx), 0)^T = \big(x-d(x,t)\bn(x,t), t \big) \quad \text{for all}~~\bx=(x,t) \in U.
\]
For $\delta$ sufficiently small, namely $\delta\le\kappa_{\max}^{-1}$,   the  decomposition $\bx=\bp(\bx)+ d(\bx) \big(\bn(\bx), 0\big)$ is unique for all $\bx \in U$  (\cite{GT}, Lemma 14.16).

The extension operator is defined as follows. For a  function $v$ on $\Gs$
we define
\begin{equation} \label{defext}
 v^e(\bx): = v(\bp(\bx)) \quad \text{for all}~~\bx \in U,
\end{equation}
i.e., $v$ is extended along \emph{spatial} normals on $\Gs$.

We need a few relations between surface norms of a function and volumetric norms of its extension.
Define $
 \mu(\bx):= \big(1-d(\bx) \kappa_1(\bx)\big)\big(1-d(\bx) \kappa_2(\bx)\big)$ for $ \bx \in U$.
 From (2.20), (2.23) in \cite{Demlow06} we have
\[
 \mu(\bx) \rd x = \rd s(\bp(\bx)) \, \rd r \, \quad \bx \in U,
\]
where $\rd x$ is the volume measure in $\Bbb{R}^3$,  $\rd s$ the surface measure on $\Gamma(t)$, and $r$ the local coordinate at $y \in \Gamma(t)$ in the (orthogonal) direction $\bn_\Gamma(y)$. Assume $\delta \le \frac14\kappa_{\max}^{-1}$. Using the relation $\kappa_i(\bx)=\frac{\kappa_i(\bp(\bx))}{1+d(\bx)\kappa_i(\bp(\bx))}$, $i=1,2$, $\bx \in U$, ((2.5) in \cite{Demlow06}) one obtains
$\frac{9}{16} \leq \mu(\bx) \leq \frac{25}{16}$ for all $\bx \in U$. Now let $v$ be a function defined on $\Gs$ and $w$, defined on $U$, given by $w(\bx)=g(\bx) v(\bp(\bx))$, with a function $g$ that  is bounded on $U$: $\|g\|_{L^\infty(U)} \leq c_g < \infty$. An example is the pair $w=v^e$ and $v$ given in \eqref{defext}, with $g \equiv 1$.  For  $v,w$ we have the following, with $U(t)=\{\, x \in  \R^{3}~|~{\rm dist}(x,\Gamma(t)) < \delta\, \}$ the cross-section of
$U$ for some $t\in[0,T]$ and a local coordinate system denoted by $\bx =(\bp(\bx),r)$:
\begin{equation} \label{normrelation}
\begin{split}
 &\|w\|_{L^2(U)}^2   = \int_U w^2(\bx) \, \rd \bx  \leq c  \int_0^T \int_{U(t)} w(\bx)^2 \mu(\bx) \, \rd x \rd t\\
 & \leq c\, \int_0^T \int_{U(t)} v(\bp(\bx))^2 \mu(\bx) \, \rd x \rd t  = c \,\int_0^T \int_{-\delta}^{\delta} \int_{\Gamma(t)} v(\bp(\bx))^2 \, \rd s(\bp(\bx))  \rd r \rd t
 \\ & \leq
  c \, \delta  \int_0^T  \int_{\Gamma(t)} v^2 \, \rd s \rd t
 \leq c \delta  \|v\|_{L^2(\Gs)}^2.
\end{split}
\end{equation}
The constant $c$ in the estimate above depends only on the smoothness of $\Gs$ 
and on $c_g$. If in addition  $|g(\bx)| \geq c_0 >0$ on $U$ holds,  then we obtain the  estimate $ \|w\|_{L^2(U)}^2 \geq c \delta  \|v\|_{L^2(\Gs)}^2$, with a constant $c >0 $ depending only on $|V_\Gamma|$ and $c_0$.
Using these results applied to $w=v^e$ as in \eqref{defext} (i.e., $g \equiv 1)$, we obtain the equivalence
\begin{equation}\label{L2equiv}
\|u^e\|_{L^2(U)}^2\simeq \delta  \|u\|_{L^2(\Gs)}^2\qquad\text{for all}~~u\in L^2(\Gs).
\end{equation}
In the remainder of this section, for $u$ defined on $\Gs$, we derive bounds on derivatives  of $u^e$ on $U$ in terms of the derivatives of $u$ on $\Gs$. We first recall a few elementary results.
From
\[
 \nabla_\Gs u = (\bI_{4 \times 4}-\bn_\Gs \bn_\Gs^T)\begin{pmatrix} \nabla_x u^e \\ u_t^e \end{pmatrix}, \quad \nabla_{\Gamma(t)} u= (\bI_{3 \times 3}- \bn_\Gamma \bn_\Gamma^T) \nabla_x u^e,
\]
one derives the following relations between tangential derivatives:
\begin{align}
 \nabla_{\Gamma(t)} u & = \bB \nabla_\Gs u, \quad \bB:=[\bI_{3 \times 3},-V_\Gamma \bn_\Gamma] \in \Bbb{R}^{3 \times 4}, \label{tanggrad1} \\
 \dot u & = (1+V_\Gamma^2) (\nabla_\Gs u)_4 +\bw \cdot \nabla_{\Gamma(t)} u,  \label{tanggrad2}
\end{align}
 where $(\nabla_\Gs u)_4 $ denotes the fourth entry of the vector $\nabla_\Gs u \in \Bbb{R}^4 $.
The spatial derivatives of the extended function can be written in terms of surface gradients (cf., e.g. (2.13) in \cite{Demlow06}):
\begin{equation}\label{gamma_grad}
\nabla_x u^e(\bx) =(\bI-d\bH)\nabla_{\Gamma(t)}u(\bp(\bx))= (\bI-d \bH) \bB \nabla_\Gs u(\bp(\bx))
=: \bB_1 \nabla_\Gs u(\bp(\bx)),
\end{equation}
for $\bx \in U$.
This implies $\nabla_x u^e(\bx)=\nabla_{\Gamma(t)}u(\bp(\bx))=\nabla_{\Gamma(t)}u(\bx)$ for $\bx \in \Gs$.
For the time derivative we obtain
\begin{multline}\label{t_grad}
u^e_t(\bx) = \frac{\partial}{\partial t} (u^e \circ\bp)(\bx) =\frac{\partial}{\partial t} u^e(x- d(x,t)\bn(x,t), t) \\   =  u^e_t(\bp(\bx))-(d_t\bn+d\bn_t)\cdot\nabla_x u^e(\bp(\bx))  =  u^e_t(\bp(\bx))-(d_t\bn+d\bn_t)\cdot\nabla_{\Gamma(t)} u(\bp(\bx)).
\end{multline}
The time derivative  $u^e_t$  on $\Gs$ can be represented in  terms of surface quantities, cf. \eqref{tanggrad2} :
\begin{equation*}
u^e_t=\dot{u}-\bw\cdot\nabla_x u^e=\dot{u}-\bw\cdot\nabla_{\Gamma(t)} u
=(1+V_\Gamma^2) (\nabla_\Gs u)_{4}\quad
\text{on}~\Gs.
\end{equation*}
Using this and \eqref{tanggrad1} in \eqref{t_grad} we obtain, for $\bx \in U$,
\begin{equation} \label{pp}
u^e_t(\bx) =(1+V_\Gamma^2) (\nabla_\Gs u(\bp(\bx)))_{4}-(d_t\bn+d\bn_t)\cdot\bB \nabla_\Gs u(\bp(\bx)) =: \bB_2 \cdot \nabla_\Gs u(\bp(\bx)).
\end{equation}
The matrices $\bB_1$, $\bB_2$ in \eqref{gamma_grad}, \eqref{pp} depend only on geometric quantities related to $\Gs$ ($d$, $d_t$, $\bH$, $V_\Gamma$, $\bn$, $\bn_t$). These quantities are uniformly bounded on $U$ due to the smoothness assumption on $\Gs$. Hence, from \eqref{gamma_grad} and the result in \eqref{normrelation} we obtain
\begin{equation}\label{H1equiv}
\|\nabla u^e\|_{L^2(U)}^2\le c\,\delta  \|\nabla_\Gs u\|_{L^2(\Gs)}^2\quad \text{for all}~u\in H^1(\Gs).
\end{equation}
We need a similar result for the $H^2$ volumetric and surface norms.    From \eqref{gamma_grad} we get $\frac{\partial u^e}{\partial x_i}(\bx)= \mathbf{b}_i \cdot \nabla_\Gs u(\bp(\bx))$, $x \in U$, $i=1,2,3$, with $\mathbf{b}_i$ the $i$-th row of the matrix $\bB_1$. For $z \in \{x_1,x_2,x_3,t\}$ we get
\[
  \frac{\partial^2 u^e}{\partial z \partial x_i}(\bx)= (\mathbf{b}_i)_z\cdot  \nabla_\Gs u(\bp(\bx)) + \mathbf{b}_i (  \nabla_\Gs \nabla_\Gs u)(\bp(\bx)) \frac{\partial}{\partial z} \bp(\bx), \quad \bx \in U.
\]
Due to the smoothness assumption on $\Gs$ the vectors $\mathbf{b}_i$, $(\mathbf{b}_i)_z$, $ \frac{\partial}{\partial z}\bp(\bx)$ have bounded $L^\infty$ norms on $U$ and application of \eqref{normrelation} yields
\begin{equation*} 
\left\| \frac{\partial^2 u^e}{\partial z \partial x_i} \right\|_{L^2(U)}^2
\le c  \delta \Big(
\sum_{|\mu|=2} \|\mathrm{D}^{\mu}_{\Gs}u\|_{L^2(\Gs)}^2+ \|\nabla_{\Gs}
u\|_{L^2(\Gs)}^2 \Big).
\end{equation*}
With similar arguments, using \eqref{pp}, one can derive the same bound for $\| \frac{\partial^2 u^e}{\partial z \partial t} \|_{L^2(U)}^2$. Hence we conclude
\begin{equation}\label{H2equiv}
\|u^e\|_{H^2(U)}^2  \leq c \delta\|u\|_{H^2(\Gs)}^2 \quad \text{for all}~~u \in H^2(\Gs).
\end{equation}

\subsection{Interpolation error bounds}\label{s_interp}
In this section, we introduce and analyze an interpolation operator. Recall that the local space-time triangulation $\mathcal{Q}_h^\Gs$ consists of cylindrical elements that are intersected by $\Gs$, cf. Fig.~\ref{Fig1}, and that the domain formed by these prisms is denoted by $Q^\Gs$.
For $K \in \mathcal{Q}_h^\Gs$, the nonempty intersections are denoted by $\Gs_K=K \cap \Gs$.
Let
\[
  I_h: C(Q^{\Gs}) \to V_{h}
\]
be the nodal interpolation operator.
Since the triangulation may vary from  time-slab to time-slab, the interpolant  is in general discontinuous between the time-slabs.

\emph{In the remainder we take} $\Delta t \sim h$. This assumption is made to avoid  anisotropic
 interpolation estimates, which would significantly complicate the analysis for the case of surface finite elements. 

We take a fixed neighborhood $U$ of $\Gs$ as in  \eqref{defU}, with $\delta >0$ sufficiently small such that the analysis presented in section~\ref{sec_ext} is valid ($\delta \leq \frac14 \kappa_{\max}^{-1}$).
 The mesh is assumed to be fine enough to resolve the geometry of $\Gs$ in the sense that  $\Q_h^{\Gs}\subset U$. We need one further technical assumption, which  holds if the space time manifold $\Gs$ is sufficiently resolved by the outer (local) triangulation  $\Q_h^{\Gs}$.
\begin{assumption} \label{assHans} \it
   For  $\Gs_K= K \cap \Gs$,  $K \in \Q_h^{\Gs}$, we assume that there is a local orthogonal coordinate system $y=(z,\theta)$, $z \in \Bbb{R}^3$, $\theta \in \Bbb{R}$, such that $ \Gs_K$ is the graph of a $C^1$ smooth scalar function, say $g_K$, i.e., $\Gs_K= \{\, (z,g_K(z))~|~z \in Z_K \subset \Bbb{R}^3\,\}$. The derivatives $\|\nabla g_K\|_{L^\infty( Z_K)}$ are assumed to be uniformly bounded with respect to $K \in \Q_h^{\Gs}$ and $h$ . Finally it is assumed that the graph $\Gs_K$ either coincides with one of the three-dimensional faces of $K$ or it subdivides $K$ into exactly two subsets (one above and one below the graph of $g_K$).
\end{assumption}
\smallskip

The next lemma is essential for our analysis of the interpolation operator. This result was presented in \cite{Hansbo02,Hansbo03}. We include a proof because the 4D case is not discussed in \cite{Hansbo02,Hansbo03}.

\begin{lemma} \label{basiclemma}
 There is a constant $c$, depending only on the shape regularity of the tetrahedral triangulations $\T_n$ and the smoothness of $\Gs$, such that
\begin{equation} \label{resl}
  \|v\|_{L^2(\Gs_K)}^2 \leq c  (h^{-1}\|v\|_{L^2(K)}^2+ h\|v\|_{H^1(K)}^2) \quad \text{for all}~~v \in H^1(K),~  K \in \Q_h^{\Gs}.
\end{equation}
\end{lemma}
\begin{proof} We recall the following trace result (e.g. Thm. 1.1.6 in \cite{BrennerScott}) for a reference simplex $\widehat{K}$:
\begin{equation*} 
 \|v\|_{L^2(\partial \widehat{K})}^2 \leq c \|v\|_{L^2(\widehat{K})} \|v\|_{H^1(\widehat{K})} \quad \text{for all} ~v \in H^1(\widehat{K}).
\end{equation*}
The Cauchy inequality and the standard scaling argument yield for $K \in \Q_h^{\Gs}$
\begin{equation} \label{traceomega}
  \|v\|_{L^2(\partial K)}^2 \leq c  (h^{-1}\|v\|_{L^2(K)}^2+ h\|v\|_{H^1(K)}^2) \quad \text{for all}~~v \in H^1(K),
\end{equation}
with a constant $c$ that depends only on the shape regularity of $K$.
 Take $K \in\Q_h^{\Gs}$ and let
$\Gs_K= \{ \, (z,g(z))~|~z \in Z_K \subset \Bbb{R}^3\, \}$ be as in Assumption~\ref{assHans}. If $\Gs_K$ coincides with one of the three-dimensional faces of $K$ then \eqref{resl} follows from \eqref{traceomega}. We consider the situation that the graph $\Gs_K$ divides $K$ into two nonempty subdomains $K_i$, $i=1,2$. Take $i$ such that $\Gs_K \subset \partial K_i$. Let $\bn= (n_1, \ldots,n_4)^T$ be the unit outward pointing normal on $\partial K_i$. For $ v \in H^1(K)$ the following holds, where ${\rm div}_y$ denotes the divergence operator in the $y=(z,\theta)$-coordinate system (cf. Assumption~\ref{assHans}),
\begin{align*}
 2 \int_{K_i} v \frac{\partial v}{\partial \theta} \, dy & = \int_{K_i} {\rm div}_y \begin{pmatrix} 0  \\  v^2 \end{pmatrix}
\, dy = \int_{\partial K_i}
\bn \cdot \begin{pmatrix} 0  \\ v^2 \end{pmatrix} \, ds = \int_{\partial K_i} n_{4} v^2 \, ds \\
& = \int_{\Gs_K}n_{4} v^2 \, ds + \int_{\partial K_i \setminus \Gs_K } n_{4} v^2 \, ds.
\end{align*}
On $\Gs_K$ the normal $\bn$ has direction $(- \nabla_z g(z), 1)^T$ and thus $n_{4}(y)=(\|\nabla_z g(z)\|^2 +1)^{-\frac12}$ holds. From Assumption~\ref{assHans}
it follows that there is a generic constant $c$ such that $1 \leq n_{4}(z)^{-1} \leq c$ holds. Using this we obtain
\begin{align*}
 & \int_{\Gs_K} v^2 \, ds  \leq c \int_{\Gs_K} n_{4} v^2 \, ds \leq c \|v\|_{L^2(K_i)} \|v\|_{H^1(K_i)} + c \int_{\partial K_i \setminus \Gs_K}  v^2 \, ds \\
& \leq c \|v\|_{L^2(K)} \|v\|_{H^1(K)}
+ c \int_{\partial K}  v^2 \, ds \\
& \leq c (h^{-1}\|v\|_{L^2(K)}^2+ h \|v\|_{H^1(K)}^2)
+ c \int_{\partial K}  v^2 \, ds \leq c  (h^{-1}\|v\|_{L^2(K)}^2+ h \|v\|_{H^1(K)}^2),
\end{align*}
where in the last inequality we used \eqref{traceomega}.
\end{proof}
\ \\[1ex]
We prove  the following approximation result:
\begin{theorem} \label{thminter}
 For sufficiently smooth $u$ defined on  $\Gs$ we have:
\begin{equation} \label{reskk1}
\begin{aligned}
 \sum_{n=1}^N\|u- I_hu^e\|_{H^k(\Gs^n)}^2 &\leq c h^{2(2-k)} \|u\|_{H^2(\Gs)}^2,  \quad k=0,1,\\
 \|u- (I_hu^e)_{-}\|_{t^n} &\leq c h^{2} \|u\|_{H^2(\Gamma(t^n))},~~ n=1,\dots,N,\\
 \|u- (I_hu^e)_{+}\|_{t^n}&\leq c h^{2} \|u\|_{H^2(\Gamma(t^n))},~~ n=0,\dots,N-1.
\end{aligned}
\end{equation}
The constants $c$ are independent of $u, h, N$.
\end{theorem}
\begin{proof} {Since $\Gs$ is a smooth three-dimensional manifold, the embedding  $H^2(\Gs)\hookrightarrow C(\Gs)$
holds. Hence $u\in C(\Gs)$ implies $ u^e\in C(U)$, and the nodal interpolant $I_h u^e$ is well defined.}
Define $v_h=(I_h u^e)|_\Gs \in W_h$. Using Lemma~\ref{basiclemma}, we obtain for $K\in \Q_h^{\Gs}$:
\[
   \|u- v_h\|_{L^2(\Gs_K)}^2 \leq c (h^{-1} \|u^e- I_hu^e\|_{L^2(K)}^2 +  h \|u^e- I_hu^e\|_{H^1(K)}^2).
\]
Standard  interpolation error bounds for $I_h$ and summing over all $ K\in \Q_h^{\Gs}$ yields
\[
   \|u- v_h\|_{L^2(\Gs)}^2 \leq c h^3 \|u^e\|_{H^2(\Q_h^{\Gs})}^2.
\]
We use $\Q_h^{\Gs}\subset U$ and \eqref{H2equiv} to infer
\[
   \|u- v_h\|_{L^2(\Gs)}^2 \leq c \delta h^3 \|u\|_{H^2(\Gs)}^2.
\]
Since we may assume $\delta\simeq h$, the result in \eqref{reskk1} follows for $k=0$. The same technique is applied to show the result for $k=1$:
\begin{align*}
  &\|\nabla_{\Gs}(u- v_h)\|_{L^2(\Gs_K)}^2  \leq c \|\nabla(u^e - I_hu^e) \|_{L^2(\Gs_K)}^2 \\
 & \leq c (h^{-1} \|\nabla(u^e- I_hu^e)\|_{L^2(K)}^2 +  h |\nabla(u^e- I_h u^e)|_{H^1(K)}^2)
  \leq c h \|u^e\|_{H^2(K)}^2.
\end{align*}
Summing over all $ K\in \Q_h^{\Gs}$ and using \eqref{H2equiv}, with  $\delta\simeq h$, then yields
the first estimate in \eqref{reskk1}. The second and third estimates follow by similar arguments, using that $u^e$ is the extension in normal \textit{spatial} direction and combining this with the \emph{three}-dimensional version of Lemma~\ref{basiclemma} and standard interpolation error bounds for $I_h u^e_{|T}$, with $T$ a tetrahedron such that $K=T\times I_n \in \Q_h^{\Gs}$.
\end{proof}

\subsection{Discretization error bound}
The next theorem is the first main result of this paper. It shows optimal convergence in the $\enorm{\cdot}_h$ norm.

\begin{theorem} \label{thmmain1}
Let $u \in \Wo$ be the solution of \eqref{weakformu} and assume $u \in  H^2(\Gs)$, $u\in H^2(\Gamma(t))$
for all $t\in[0,T]$. 
Let $u_h \in W_h$ be the solution of the discrete problem \eqref{brokenweakformu_h1} with a stabilization parameter $\sigma$ as in Theorem~\ref{stabadd}.  The following error bound holds:
\[
 \enorm{u-u_h}_h \leq c h (\|u\|_{H^2(\Gs)}+\sup_{t\in[0,T]}\|u\|_{H^2(\Gamma(t))}).
\]
\end{theorem}
\begin{proof}
For the solution $u \in  H^2(\Gs)$ let $e_I= u-(I_h u^e)|_{\Gs}$ denote the interpolation error and $e=u-u_h$ the discretization error. The inf-sup stability result in \eqref{infsuph} with $W^b$ replaced by $W_h$ and the continuity result \eqref{cont}
imply in a standard way, cf. e.g. \cite{Ern04}:
\[
\enorm{e}_h\le \enorm{e_I}_h+ c (\|e_I\|_{W^b}+\sum_{n=0}^{N-1}\|[e_I]^n\|_{t^n}).
\]
Using the first interpolation  bound in Theorem~\ref{thminter} and $H_n \subset L^2(\Gs^n)$ we get
\begin{equation} \label{estu1}
\begin{split}
 \|e_I\|_{W^b}^2 &=  \sum_{n=1}^N\|(\dot{e}_I)_n\|_{H'_n}^2 + \|e_I\|_H^2 \le  \sum_{n=1}^N \|(\dot{e}_I)_n\|_{L^2(\Gs^n)}^2 + \|e_I\|_H^2  \\ & \leq c \sum_{n=1}^N\|(e_I)_n\|_{H^1(\Gs^n)}^2
  \leq c h^2 \|u\|_{H^2(\Gs)}^2.
  \end{split}
\end{equation}
Furthermore, applying the result in the second and the third interpolation  bounds in Theorem~\ref{thminter} we obtain
\[
\begin{split}
\sum_{n=0}^{N-1}\|[e_I]^n\|_{t^n}&\le \|(e_I)_{+}\|_{t^0}+ \sum_{n=1}^{N-1}(\|(e_I)^n_{-}\|_{t^n}+\|(e_I)^n_{+}\|_{t^n})\\ &\le
c\,h^2\,(\Delta t)^{-1}\sup_{n=0,\dots,N-1}\|u\|_{H^2(\Gamma(t^n))}\le c\,h\,\sup_{t\in[0,T]}\|u\|_{H^2(\Gamma(t))}.
  \end{split}
\]
 This together with \eqref{estu1} proves the theorem.
 \end{proof}

\section{Second  order convergence} \label{sec7}
The aim of this section is to derive an error estimate $\|u-u_h\|_\ast \leq c h^2$ for $\Delta t \sim h$ in a suitable norm with the help of a duality argument.
To formulate an adjoint problem, we define a ``reverse time'' in the space-time manifold $\Gs$.
Let $X(t)$ be the Lagrangian particle path given by $\bw$ and initial manifold $\Gamma_0$:
\[
  \frac{d X}{dt}(t)= \bw (X(t),t), \quad t \in [0,T],~~X(0) \in \Gamma_0.
\]
Hence, $\Gamma(t)=\{\, X(t)~|~X(0) \in \Gamma_0\,\}$. Define, for $t \in [0,T]$:
\[
  \widetilde X(t):= X(T-t),~~\widetilde \Gamma(t):=\Gamma(T-t),~~\widetilde \bw(x,t):=-\bw(x,T-t), \quad x \in \Omega.
\]
From
\[
  \frac{d\widetilde X}{dt}(t)= - \frac{dX}{dt}(T-t)= -\bw (X(T-t),T-t)= \widetilde \bw (\widetilde X(t),t),
\]
it follows that
$\widetilde X(t)$ describes the particle paths corresponding to the flow $\widetilde \bw$ with $\widetilde X(0) = X(T) \in \Gamma(T)$.
Hence, $\widetilde \Gamma(t)= \{\, \widetilde X(t)~|~ \widetilde X(0) \in \Gamma(T)=\widetilde \Gamma_0\, \}$. We introduce the material derivative with respect to the flow field $\widetilde \bw$:
\[
 \check{v}(x,t):= \frac{\partial v}{\partial t}(x,t)+ \widetilde \bw (x,t) \cdot \nabla v(x,t), \quad (x,t) \in \Gs.
\]
For a given $f^\ast \in L^2(\Gs)$ we consider the following \emph{dual problem}
\begin{equation} \label{dual}
 \begin{split}
  \check{v} - \nu_d \Delta_{\widetilde \Gamma} v + \sigma \int_{\widetilde\Gamma(t)} v\, ds & = f^\ast \quad \text{on}~~\widetilde{\Gamma}(t), ~t \in [0,T], \\
    v(\cdot,0) &= 0 \quad \text{on}~~\widetilde \Gamma_0=\Gamma(T).
 \end{split}
\end{equation}
The problem \eqref{dual} is of integro-differential type. From the analysis of \cite{ORXsinum} it follows that a weak formulation of this problem as in \eqref{weakformu}, with the bilinear form
$a(\cdot,\cdot)$ replaced by $a_\sigma(\cdot,\cdot)$, has a unique solution $v \in \Wo$.
As is usual in the Aubin-Nitsche duality argument, we need a suitable regularity result for the dual problem \eqref{dual}. In the literature we did not find the regularity result that we need. Therefore we derived the result given in the following theorem. A proof is given in the next section. A corollary of this theorem gives the regularity result for the dual problem that we need.

\begin{theorem} \label{Th_regul}
Consider the parabolic surface problem
\begin{equation} \label{orig_A}
 \begin{split}
  \dot{u} - \nu_d \Delta_{\Gamma} u & = f \quad \text{on}~~\Gamma(t),~t \in (0,T], \\
    u(\cdot,0) &= 0 \quad \text{on}~~\Gamma_0,
 \end{split}
\end{equation}
Let $\Gs$ be sufficiently smooth (precise assumptions are given in the proof) and $f\in L^2(\Gs)$.
 Then the unique weak solution $u\in \Wo$ of \eqref{orig_A}
satisfies $u \in H^1(\Gs)$,  $u\in H^2(\Gamma(t))$ for almost all  $t\in [0,T]$, and
\begin{equation} \label{fg}
\|u\|_{H^1(\Gs)}^2 +\int_{0}^T\|u\|^2_{H^2(\Gamma(t))}\, d t\le c  \|f\|^2_0,
\end{equation}
with a constant $c$ independent of $f$.
If in addition $f \in H^1(\Gs)$ and $f|_{\Gamma_0}=0$, then $u \in H^2(\Gs)$ and
\begin{equation} \label{H2reg}
\sup_{t\in[0,T]}\|u\|_{H^2(\Gamma(t))}+ \|u\|_{H^2(\Gs)}\leq c \|f\|_{H^1(\Gs)},
\end{equation}
with a constant $c$ independent of $f$.
\end{theorem}
 \medskip

\begin{corollary} \label{Cl_regul}Let $\Gs$ be sufficiently smooth (as in Theorem~\ref{Th_regul}).
Assume $f^\ast \in H^1_0(\Gs)$. Then the unique weak solution $v\in W_0$
of \eqref{dual} satisfies $v \in H^2(\Gs)$  and
\begin{equation}\label{H2est}
\sup_{t\in[0,T]}\|v\|_{H^2(\Gamma(t))}+ \|v\|_{H^2(\Gs)}\leq c \|f^\ast\|_{H^1(\Gs)},
\end{equation}
with a constant $c$ independent of $f^\ast$.
\end{corollary}
\begin{proof}
We have $v\in W_0\subset L^2(\Gs)$. Hence, $\int_{\widetilde\Gamma(t)} v\, ds\in L^2(\Gs)$
and
\[\left\|\int_{\widetilde\Gamma(t)} v\, ds\right\|_0\le (\max_{t\in[0,T]}|\widetilde\Gamma(t)|)\|v\|_0\le c\,\|f^\ast\|_{H'}
\le c\,\|f^\ast\|_{0}.\]
Therefore, $v$ solves the  parabolic surface problem
\begin{equation*} 
 \begin{split}
  \check{v} - \nu_d \Delta_{\widetilde \Gamma} v & = F \quad \text{on}~~\widetilde{\Gamma}(t), \\
    v(\cdot,0) &= 0 \quad \text{on}~~\widetilde \Gamma_0,
 \end{split}
\end{equation*}
with  $F:=f^\ast-\sigma\int_{\widetilde\Gamma(t)} v\, ds \in  L^2(\Gs)$  and $\|F\|_0\leq c \|f^\ast\|_0$. The first part of Theorem~\ref{Th_regul} yields   $\check{v}\in L^2(\Gs)$ and $\|\check{v}\|_0\leq c \|F\|_0$. Hence, employing   the Leibniz
formula we check $\frac{\partial }{\partial t}\int_{\widetilde\Gamma(t)} v\, ds\in L^2(\Gs)$. This
and $v\in H$ yields $\int_{\widetilde\Gamma(t)} v\, ds\in H^1(\Gs)$ together with a
corresponding a priori estimate. Therefore, $F\in H^1(\Gs)$ and $\|F\|_{H^1(\Gs)}\le c\,\|f^\ast\|_{H^1(\Gs)}$. From $v(\cdot,0)=0$ on $\widetilde \Gamma_0$ and $f^\ast|_{\widetilde \Gamma_0}=0 $  we get $F|_{\widetilde \Gamma_0}=0$. Applying the second part of the theorem completes the proof.
\end{proof}
\medskip

\begin{lemma} \label{vardual}
Assume  $v\in H^2(\Gs)$ solves \eqref{dual} for some $f^\ast \in H^1_0(\Gs)$.
Define
$v^\ast(x,t):=v(x,T-t), \quad x \in  \Gamma(t)=\widetilde \Gamma(T-t)$.
Then one has
\begin{equation} \label{dual2}
 \la \dot z, v^\ast \ra_b +a_\sigma(z,v^\ast) +d(z,v^\ast)= (z,f^\ast)_0 \quad \text{for all}~~z \in W_h + H^1(\Gs).
\end{equation}
\end{lemma}
\begin{proof}
 From the definitions and using Leibniz rule we obtain (note that $v^\ast $ is continuous, hence $v_{-}^{\ast,n}=v_{+}^{\ast,n}=v^{\ast,n}$):
\begin{align*}
   &\la \dot z, v^\ast \ra_b +a_\sigma(z,v^\ast) +d(z,v^\ast) \\
 &= \sum_{n=1}^N \int_{t_{n-1}}^{t_n} \int_{\Gamma(t)} \dot z v^\ast + z v^\ast \DivG \bw \, ds \, dt +\sum_{n=1}^N ([z]^{n-1},v^{\ast,n-1})_{t_{n-1}}\\
 &\quad+\nu_d (\nablaG z, \nablaG v^\ast)_0
 +\sigma\int_0^T\int_{\Gamma(t)} z\,dx\int_{\Gamma(t)} v^\ast\,dx\,dt\\
& = \sum_{n=1}^N \big( (z_{-}^n,v^{\ast,n})_{t_n}-(z_{+}^{n-1},v^{\ast,n-1})_{t_{n-1}}\big) - \sum_{n=1}^N \int_{t_{n-1}}^{t_n} \int_{\Gamma(t)} z \dot v^\ast \, ds \, dt  \\  & \quad + \sum_{n=1}^N ([z]^{n-1},v^{\ast,n-1})_{t_{n-1}} +\nu_d (\nablaG z, \nablaG v^\ast)_0+\sigma(z, \int_{\Gamma(t)}v^\ast\,dx)_0\\
& = -(\dot v^\ast +  \nu_d \Delta_\Gamma v^\ast-\sigma\int_{\Gamma(t)} v^\ast\,dx, z)_0 .
\end{align*}
Now note that on $\Gs$:
\[
\begin{split}
 \dot v^\ast(\cdot,t) & =\frac{\partial v^\ast}{\partial t}(\cdot,t)+ \bw(\cdot,t)\nabla v^\ast(\cdot,t)= -\frac{\partial v}{\partial t}(\cdot,T-t)-\widetilde\bw (\cdot, T-t)\cdot \nabla v(\cdot, T-t)  \\ & = - \check{v}(\cdot,T-t),
\end{split} \]
 and $\Delta_{\Gamma(t)}v^\ast(\cdot,t)=\Delta_{\widetilde \Gamma(T-t)} v(\cdot,T-t)$. From this and the equation for $v$ in \eqref{dual} it follows that $\dot v^\ast +  \nu_d \Delta_\Gamma v^\ast-\sigma\int_{\Gamma(t)} v^\ast\,dx=-f^\ast $ on $\Gs$. This completes the proof.
\end{proof}

Denote by $\|\cdot\|_{-1}$ a norm dual to the $H^1_0(\Gs)$ norm with respect to the $L^2$-duality.
In the  next theorem we present the second main result of this paper.

\begin{theorem} \label{dualthm}  Assume that $\Gs$ is sufficiently smooth (as in Theorem~\ref{Th_regul}) and that the assumptions of Theorem~\ref{thmmain1} are satisfied.
Then the following  error estimate holds:
\[
 \|u-u_h\|_{-1} \leq c h^2 (\|u\|_{H^2(\Gs)}+\sup_{t\in[0,T]}\|u\|_{H^2(\Gamma(t))}).
\]
\end{theorem}
\begin{proof}
Take arbitrary $f^\ast\in H^1_0(\Gs)$. Using the relation in \eqref{dual2}, Galerkin orthogonality, the second continuity result in Lemma~\ref{Lem_cont} and the error estimate from Theorem~\ref{thmmain1}  we obtain with
$e:=u-u_h$, $e_I=v^\ast -I_h(v^\ast)^e \in W^b$:
\begin{align*}
(e,f^\ast)_0 & = \la \dot e, v^\ast \ra_b +a_\sigma(e,v^\ast) +d(e,v^\ast)
   =  \la \dot e, e_I \ra_b +a_\sigma(e,e_I) +d(e,e_I) \\& \leq c \enorm{e}_h (\|e_I\|_{W^b}+\sum_{n=1}^{N-1}\|[e_I]^n\|_{t^n}+\|e_I\|_T)   \\
  & \leq c h (\|u\|_{H^2(\Gs)}+\sup_{t\in[0,T]}\|u\|_{H^2(\Gamma(t))}) (\|e_I\|_{W^b}+\sum_{n=1}^{N-1}\|[e_I]^n\|_{t^n}+\|e_I\|_T)
 \end{align*}
Applying interpolation estimates  as in the proof of  Theorem~\ref{thmmain1}, we get
\[
\|e_I\|_{W^b}+\sum_{n=1}^{N-1}\|[e_I]^n\|_{t^n}+\|e_I\|_T\le c\,h\,(\|v^\ast\|_{H^2(\Gs)}+\sup_{t\in[0,T]}\|v^\ast\|_{H^2(\Gamma(t))}).
\]
Hence, using \eqref{H2est} we get
\begin{align*}
(e,f^\ast)_0 &  \leq c h^2  (\|u\|_{H^2(\Gs)}+\sup_{t\in[0,T]}\|u\|_{H^2(\Gamma(t))})(\|v^\ast\|_{H^2(\Gs)}+\sup_{t\in[0,T]}\|v^\ast\|_{H^2(\Gamma(t))}) \\ & \leq c h^2  (\|u\|_{H^2(\Gs)} + \sup_{t\in[0,T]}\|u\|_{H^2(\Gamma(t))}) \|f^\ast\|_{H^1(\Gs)}.
\end{align*}
 From this the result immediately follows.
\end{proof}

\begin{remark} \label{remNorm} \rm Numerical experiments suggest that the method has second order convergence  in the  $L^2(\Gs)$ norm. 
We proved the second order convergence only in the weaker $H^{-1}(\Gs)$ norm.
The reason for using this weaker norm  is that our arguments use isotropic  polynomial interpolation error
bounds on 4D space-time elements. Naturally, such bounds require isotropic space-time $H^2(\Gs)$-regularity bounds for the  solution. For our  class of parabolic problems such isotropic regularity bounds are more restrictive than in an elliptic case, since the solution is in  general less regular in time than in space. Due to this, instead of the common $f^*\in L^2(\Gs)$ regularity assumption  for the right-hand side
of the dual problem we need the stronger assumption $f^*\in H^1(\Gs)$ to guarantee a $H^2(\Gs)$-regularity of the solution. This stronger regularity requirement for $f^\ast$ results in the weaker $H^{-1}(\Gs)$ error norm. It may be possible to derive second order convergence in the $L^2(\Gs)$-norm, if suitable anisotropic  interpolation estimates are available.   So far, however, we have not been able to derive such estimates for  the finite element space-time trace space. This topic is left for future  research.
\end{remark}

\section{Proof of Theorem~\ref{Th_regul}}
\label{Sec_A}\label{sec8}
 Without   loss of generality we may
set $\nu_d=1$. 
The weak formulation of \eqref{orig_A} is as follows: determine $u \in \Wo$ such that
\begin{equation} \label{weakformuP}
 \la\dot u,v\ra +(\nablaG u, \nablaG v)_0 =( f,v)_0 \quad \text{for all}~~v \in H.
\end{equation}
The proof is based on techniques as in \cite{Dziuk07}, \cite{Evans}. We define a Galerkin solution in a sequence of nested spaces spanned by a special
choice of smooth basis functions. We derive uniform energy estimates for these Galerkin solutions and based on a compactness argument  these estimates imply a bound in  the $\|\cdot\|_{H^1(\Gs)}$ norm for the weak  limit of these Galerkin solutions. We use a known $H^2$-regularity result for the Laplace-Beltrami equation on a smooth manifold and energy estimates for the material derivative of the Galerkin solutions to derive a bound on the $\|\cdot\|_{H^2(\Gs)}$ norm for the weak  limit of these Galerkin solutions.\\[1ex]
\emph{1. Galerkin subspace and boundedness of $L^2$-projection}.
We introduce Galerkin subspaces of $\Wo$, similar to those used in  \cite{Dziuk07}. For this we
 need a smooth diffeomorphism between $\Gs$ and the cylindrical  reference domain $\widehat{\Gs}:=\Gamma_0\times(0,T)$.
 We use a Langrangian mapping from  $\Gamma_0 \times [0,T]$ to the space-time manifold  $\Gs$, as in  ~\cite{ORXsinum}. The velocity field $\bw$ and  $\Gamma_0$ are sufficiently smooth such that for all $y \in \Gamma_0$ the ODE system
\[
  \Phi(y,0)=y, \quad \frac{\partial \Phi}{\partial t}(y,t)= \bw(\Phi(y,t),t), \quad t\in [0,T],
\]
has a unique solution $x:=\Phi(y,t) \in \Gamma(t)$ (recall that $\Gamma(t)$ is transported with the  velocity field $\bw$). The corresponding inverse mapping is given by $\Phi^{-1}(x,t):=y \in \Gamma_0$, $x \in \Gamma(t)$. The Lagrangian mapping  $\Phi$  induces a bijection
\[ F:\,   \Gamma_0 \times [0,T] \to \Gs, \quad ~F(y,t):=(\Phi(y,t),t).
\]
  We assume this bijection to be a $C^2$-diffeomorphism between these manifolds.

For a function $u$ defined on $\Gs$  we define $\widehat u=u\circ F$ on $\Gamma_0 \times (0,T)$:
\[
 \widehat u(y,t)=u(\Phi(y,t),t)=u(x,t).
\]
Vice versa, for a function $\widehat u$ defined on $\Gamma_0 \times (0,T)$  we define $u=\widehat u\circ F^{-1}$ on $\Gs$:
\begin{equation*}
u(x,t)=\widehat u(\Phi^{-1}(x,t),t)=\widehat u(y,t).
\end{equation*}
By construction, we have
\begin{equation} \label{transf1}
  \dot u (x,t)= \frac{\partial \widehat u}{\partial t}(y,t).
\end{equation}
We need a surface integral transformation formula. For this we consider  a local parametrization of $\Gamma_0$, denoted by $\mu:\Bbb{R}^2 \to \Gamma_0$, which is {at least $C^2$ smooth}. Then, $\Phi \circ \mu:=\Phi(\mu(\cdot), t)$ defines a {$C^2$ smooth} parametrization of $\Gamma(t)$.    For the surface measures $d\,\widehat{s}$ and $ds$ on $\Gamma_0$ and $\Gamma(t)$, respectively, we have the relations
\begin{equation} \label{fromtrans1}
  ds =\gamma(\cdot ,t)\, d\,\widehat{s},\qquad
 d\,\widehat{s}=
  \widetilde\gamma(\cdot,t) \, ds,
\end{equation}
with functions $\gamma$ and $\widetilde\gamma$ that are both
$C^1$ smooth, bounded and uniformly bounded away from zero: $\gamma\ge c>0$ on $\Gamma_0\times (0,T)$ and
$\widetilde\gamma\ge c>0$ on ${\Gs}$, cf. section 3.3 in~\cite{ORXsinum}.

Denote by $\widehat{\phi}_j$, $j\in \mathbb{N}$ the eigenfunctions of the Laplace-Beltrami operator on $\Gamma_0$.
Define $\phi_j:\, \Gs \to \Bbb{R}$
by $
\phi_j(\Phi(y,t),t):=\widehat{\phi}_j(y),
$
and note that due to \eqref{transf1} one has $\dot \phi_j=0$.
The set $\{\phi_j(\cdot,t)\,|\,j\in\mathbb{N}\}$ is dense in $H^1(\Gamma(t))$.   We define the spaces
\begin{align*} X_N(t) & = \text{span}\{\phi_1(\cdot,t),\dots,\phi_N(\cdot,t)\}, \\
 X_N &= \{\, \sum_{j=1}^N u_j(t)\phi_j(x,t)~|~ u_j \in H^1(0,T; \Bbb{R}), ~u_j(0)=0, ~1 \leq j \leq N\,\}.
\end{align*}
Below, in step 2,  we construct a Galerkin solution in the subspace $X_N \subset \Wo$. Note that for $v \in X_N$ we have $v(\cdot,t) \in X_N(t)$. In the analysis in step 6, we need  $H^1$-stability of the  $L^2$-projection on $X_N(t)$. This stability result is derived in the following lemma.

\begin{lemma}\label{l_A2} Denote by $P_{X_N(t)}$ the $L^2$-orthogonal projector on $X_N(t)$, i.e.,  for $\zeta \in L^2(\Gamma(t))$:
\[
  \int_{\Gamma(t)} P_{X_N(t)} \zeta \, v \,ds = \int_{\Gamma(t)} \zeta v \, ds \quad \text{for all}~~v \in X_N(t).
\]
For $\zeta \in H^1(\Gamma(t))$ the estimate
\begin{equation}\label{H1stab}
\|\nabla_{\Gamma} P_{X_N(t)}\zeta\|_{L^2(\Gamma(t))}\le C\, \|\zeta\|_{H^1(\Gamma(t))}
\end{equation}
holds with a constant independent of $N$ and $t$.
\end{lemma}
\begin{proof}
{Fix some $t\in(0,T)$ and let $\gamma$ be a smooth and positive function on $\Gamma_0$ defined in \eqref{fromtrans1},}
then $(f,g)_{\gamma}:=\int_{\Gamma_0}fg\,\gamma\, d s$ defines a scalar product on $L^2(\Gamma_0)$. This scalar product induces a norm equivalent to the standard $L^2(\Gamma_0)$-norm. For given $f\in H^1(\Gamma_0)$ let $f_N$ be an $(\cdot,\cdot)_\gamma$-orthogonal
projection on $X_N(0)$.
Since $\Delta_\Gamma f_N\in X_N(0)$, we have $
\int_{\Gamma_0}\gamma \,f \Delta_\Gamma f_N \, d s=\int_{\Gamma_0}\gamma\, f_N \Delta_\Gamma f_N \, d s$.
Using this and integration by parts we obtain  the identity:
\[
\int_{\Gamma_0}|\nabla_\Gamma f_N|^2\,\gamma\, d s=\int_{\Gamma_0}(\nabla_\Gamma f_N\nabla_\Gamma \gamma)\,(f-f_N)\, d s+
\int_{\Gamma_0}(\nabla_\Gamma f_N\nabla_\Gamma f)\gamma\,\, d s.
\]
Applying the Cauchy inequality, positivity and smoothness of $\gamma$, we get
\[
\int_{\Gamma_0}|\nabla_\Gamma f_N|^2\, d s \le c\,\int_{\Gamma_0}f^2+|\nabla_\Gamma f|^2\,\, d s,
\]
i.e. the $(\cdot,\cdot)_\gamma$-orthogonal  projection on $X_N(0)$ is $H^1$-stable. For $\zeta \in H^1(\Gamma(t))$
define $\widehat{\zeta}=\zeta\circ\Phi\in H^1(\Gamma_0)$ and $\widehat{\zeta}_N=\zeta_N\circ\Phi\in X_N(0)$. From
\[
\int_{\Gamma_0}\widehat{\zeta}_N\widehat{\psi}_N\gamma\,\, d \widehat{s}=\int_{\Gamma(t)}\zeta_N\psi_N\,\, d s=
\int_{\Gamma(t)}\zeta \psi_N\,\, d s=\int_{\Gamma_0}\widehat{\zeta}\widehat{\psi}_N\gamma\,\,d \widehat{s}\quad\forall~\widehat{\psi}_N\in X_N(0),
\]
it follows that  $\widehat{\zeta}_N$ is the $(\cdot,\cdot)_\gamma$-orthogonal  projection of $\widehat{\zeta}$.
Using the $H^1$-stability of this projection, the smoothness of $\Phi$
and $\Phi^{-1}$ and \eqref{fromtrans1}, we obtain
\[
\|\nabla_{\Gamma} \zeta_N\|_{L^2(\Gamma(t))}\le C\, \|\nabla_{\Gamma} \widehat{\zeta}_N\|_{L^2(\Gamma_0)}
\le C\, \|\widehat{\zeta}\|_{H^1(\Gamma_0)} \le C\,\|\zeta\|_{H^1(\Gamma(t))}.
\]
Thus, the estimate in \eqref{H1stab} holds.
\end{proof}
\ \\[1ex]
\emph{2. Existence of Galerkin solution $u_N \in X_N$ and its boundedness in $H^1(\Gs)$ uniformly in $N$.}
We look for a Galerkin solution $u_N\in X_N$ to \eqref{orig_A}. We consider the following  projected  surface parabolic equation: determine $\bu_N =(u_1, \ldots u_N) \in H^1(0,T;\Bbb{R}^N)$ such that for $u_N(x,t):= \sum_{j=1}^N u_j(t)\phi_j(x,t)$ we have $u_N(\cdot,0)=0$ and
\begin{equation}\label{ODE}
\int_{\Gamma(t)} (\dot{u}_N -\Delta_{\Gamma} u_N)\phi \, \,d s = \int_{\Gamma(t)} f \phi \, \, d s\quad \text{for all}~~ \phi \in X_N(t), \quad \text{a.e. in}~t \in [0,T].
\end{equation}
In terms of $\bu_N$ this can be rewritten as a linear system of ODEs of the form
\begin{equation} \label{ODE1}
M(t)\frac{d\bu_N}{dt} + A(t)\bu_N(t)=b(t), \quad \bu_N(0)=0.
\end{equation}
 The matrices $M, A$  are symmetric positive semi-definite. Since for the eigenfunctions we have $\widehat{\phi}_i\in C^2(\Gamma_0)$, see \cite{Aubin},  and the diffeomorphism $F$ is $C^2$-smooth, we have  $M, A \in W^{1}_{\infty}(0,T;\Bbb{R}^{N \times N})$. The smallest eigenvalue of $M(t)$ is bounded away from zero uniformly in $t \in [0,T]$. The right-hand side satisfies $b \in L^2(0,T;\Bbb{R}^N)$.
 By the theory of linear ordinary differential equations, e.g., Proposition 6.5 in \cite{Hunter}, we have  existence of a unique solution $\bu_N \in H^1(0,T;\Bbb{R}^N)$. Moreover, if $f\in H^1(\Gs)$, then $b \in H^1(0,T;\Bbb{R}^N)$ and
$\bu_N \in H^2(0,T; \Bbb{R}^N)$.
For the corresponding Galerkin solution $u_N \in X_N$, given by $u_N(x,t) =\sum_{j=1}^N u_j(t)\phi_j(x,t)$, we derive energy estimates.
Taking $\phi=u_N(\cdot,t) \in X_N(t)$  in \eqref{ODE} and applying integration by parts we obtain the identity
\[
\frac12\frac{\, d}{\, d t}\int_{\Gamma(t)} {u}_N^2 \, \,d s + \int_{\Gamma(t)} |\nabla_\Gamma u_N|^2
-\frac12(\div_\Gamma\bw){u}_N^2 \, \, d s=\int_{\Gamma(t)} fu_N \, \, d s.
\]
Applying the Cauchy inequality to handle the term on the right-hand side
 and  using a Gronwall argument, with $u_N(\cdot,0)=0$, yields
\[
\sup_{t\in(0,T)}\int_{\Gamma(t)}{u}_N^2\, d s + \int_{0}^T\int_{\Gamma(t)}|\nabla_\Gamma u_N|^2\,d s\,d t
\le C \|f\|^2_0,
\]
and thus
\begin{equation}\label{est_energy}
 \|u_N\|_H \le C \|f\|_0,
\end{equation}
with a constant independent of $N$.
Taking $\phi=\dot u_N(\cdot,t) \in X_N(t)$  in \eqref{ODE} and using the identity \[\int_\Gamma \nablaG v \cdot \nablaG \dot v\, \,d s = \frac12 \frac{d}{dt} \int_\Gamma |\nablaG v|^2 \, \,d s -\frac12 \int_\Gamma|\nablaG v|^2 \div_\Gamma\bw\, \, d s + \int_\Gamma D(\bw) \nablaG v \cdot \nablaG v \, \,d s,\] with the  tensor $D(\bw)_{ij}=\frac12 \big( \frac{\partial \bw_j}{\partial x_i} + \frac{\partial \bw_i}{\partial x_j}\big) $ (cf. (2.11) in \cite{Dziuk07}) yields
\begin{align*}
& \int_{\Gamma(t)}\dot{u}_N^2 \, \,d s +\frac12\frac{\,d}{\,d t}\int_{\Gamma(t)} |\nabla_\Gamma u_N|^2 \, \,d s \\
  & = \frac12\int_{\Gamma(t)} |\nabla_\Gamma u_N|^2\div_\Gamma\bw \,  \,d s-
\int_{\Gamma(t)} D(\bw)\nabla_\Gamma u_N\cdot\nabla_\Gamma u_N \, \,d s
+\int_{\Gamma(t)} f\dot{u}_N\, \,d s.
\end{align*}
Employing the Cauchy inequality and a Gronwall inequality, with $u_N(\cdot,0)=0$, we obtain
\begin{equation}\label{est_energy2}
\sup_{t\in(0,T)}\int_{\Gamma(t)}|\nabla_\Gamma u_N|^2\,d s + \int_{0}^T\int_{\Gamma(t)}|\dot{u}_N|^2\,d s\,d t
\le C \|f\|^2_0,
\end{equation}
with a constant independent of $N$. From the results in \eqref{est_energy} and \eqref{est_energy2} we obtain the uniform boundedness result
\begin{equation} \label{uniH1}
 \|u_N \|_{H^1(\Gs)} \leq C \|f\|_0.
\end{equation}
\emph{3. The weak limit $u$ solves \eqref{weakformuP} and $\|u\|_{H^1(\Gs)} \leq C \|f\|_0$ holds.}
 From the uniform boundedness \eqref{uniH1} it follows that there is a subsequence, again denoted by $(u_N)_{N \in \Bbb{N}}$, that weakly converges to some $u\in H^1(\Gs)$:
\begin{equation} \label{weakl}
u_N\rightharpoonup  u\quad\text{in}~~ H^1(\Gs).
\end{equation}
As a direct consequence of this weak convergence and \eqref{uniH1} we get
\begin{equation}\label{uest1}
\|u \|_{H^1(\Gs)}
\le c \|f\|_0.
\end{equation}
We recall an elementary result from  functional analysis. Let $X$, $Y$ be normed spaces, $T: X \to Y$  linear and bounded and  $(x_n)_{n \in \Bbb{N}}$ a sequence in $X$, then the following holds:
\begin{equation} \label{elFA}
  x_n \rightharpoonup x  \quad \text{in}~X ~~ \Rightarrow ~~Tx_n \rightharpoonup Tx  \quad \text{in}~Y.
\end{equation}
Hence, from \eqref{weakl} we obtain the following, which we need further on:
\begin{equation} \label{conv2}
  \dot u_N\rightharpoonup  \dot u\quad\text{in}~~ L^2(\Gs), \quad u_N\rightharpoonup  u\quad\text{in}~~ H.
\end{equation}
We now show that $u$ is the solution of \eqref{weakformuP}. Define $\hat X_N:={\rm span}\{\hat \phi_1, \ldots, \hat \phi_N\}$ and note that $\cup_{N \in \Bbb{N}} \hat X_N $ is dense in $H^1(\Gamma_0)$. The set $\hat{C}=\{ \, t \to \sum_{j=0}^n t^j \hat \psi_j ~|~\hat \psi_j \in \hat X_N, ~n, N \in \Bbb{N} \, \}$ is dense in $L^2(0,T; H^1(\Gamma_0))$. Using this and Lemma 3.3 in \cite{ORXsinum} it follows that $C=\{ \sum_{j=0}^n t^j \psi_j(x,t) ~|~\psi_j(\cdot,t) \in  X_N(t), ~n, N \in \Bbb{N} \, \}$ is dense in $H$. Consider $\psi(x,t)=t^j \phi_k(x,t)$. From \eqref{ODE} it follows that for $N \geq k$ we have
\[
 \int_0^T\int_{\Gamma(t)} \dot{u}_N \psi + \nablaG u_N \cdot \nablaG \psi \, \,d s  \, \,d t= \int_0^T\int_{\Gamma(t)} f \psi \, \,d s \, \,d t
\]
and using \eqref{weakl} it follows that this equality  holds with $u_N$ replaced by $u$. From linearity and density of $C$ in $H$ we conclude that $u \in H^1(\Gs) \subset W$ solves \eqref{weakformuP}. It remains to check whether $u$ satisfies the homogeneous initial condition.

 From the weak convergence in $H^1(\Gs)$, the boundedness of the trace operator $T: H^1(\Gs) \to L^2(\Gamma_0)$, $Tv=v(\cdot,0)$ and \eqref{elFA} it follows that $u_N(\cdot,0)$ converges weakly to $u(\cdot,0)$ in $L^2(\Gamma_0)$. From the property $u_N(\cdot,0)=0$ for all $N$ it follows that $u(\cdot,0)=0$ holds.
Hence $u \in \Wo$ holds.\\[1ex]
\emph{4. The estimate  $\|\nablaG^2 u\|_0 \leq c \|f\|_0$ holds.} The function  $u$ is a (weak) solution of $-\Delta_{\Gamma} u =f-\dot{u}$ on $\Gamma(t)$, with $f(\cdot,t)-\dot{u}(\cdot,t)\in L^2(\Gamma(t))$ for almost all $t \in [0,T]$. The $H^2$-regularity theory for a Laplace-Beltrami equation
on a smooth manifold (see \cite{Aubin}) yields $u\in H^2(\Gamma(t))$ and
\begin{equation}\label{eqAubin}
\|u\|_{H^2(\Gamma(t))} \leq C_t \|f(\cdot,t)-\dot{u}(\cdot,t)\|_{L^2(\Gamma(t))}.
 \end{equation}
Due to the smoothness of $\Gs$ we can assume $C_t$ to be uniformly bounded w.r.t. $t$. Using this and \eqref{uest1}
we get
\begin{equation}\label{uest2}
\|\nablaG^2 u\|_0^2 \leq \int_{0}^T\|u\|^2_{H^2(\Gamma(t))}\,d t
\leq  c \int_0^T \|f(\cdot,t)-\dot{u}(\cdot,t)\|_{L^2(\Gamma(t))}^2 \, \,d t \leq c \|f\|^2_0.
\end{equation}
 From this and  \eqref{uest1} the result \eqref{fg} follows.
\\[1ex]
\emph{5. The estimate  $\sup_{t\in[0,T]}\|u\|_{H^2(\Gamma(t))} + \|\nablaG \dot u\|_0 \leq c \|f\|_{H^1(\Gs)}$ holds.}
We will use the assumptions $f\in H^1(\Gs)$ and $f|_{t=0}=0$. We need a commutation formula for the
material derivative and the Laplace-Beltrami operator. To derive this, we use the notation $\nablaG g=(\uD_1g, \ldots,\uD_d g)^T$ for the components of the tangential derivative and the following identity, given in Lemma 2.6 of \cite{DziukKroener}:
\[
  \dot{(\uD_i g)}=\uD_i \dot{g} - A_{ij}(\bw) \uD_j g, ~\text{with}~~ A_{ij}(\bw)=\uD_i \bw_j - \nu_i \nu_s \uD_j \bw_s, \quad \bn_\Gamma=(\nu_1, \ldots, \nu_d)^T.
\]
 Let $\nablaG  \bw=(\nablaG w_1 \ldots \nablaG w_d ) \in \Bbb{R}^{d \times d}$, $\bA= \nablaG \bw - \bn_\Gamma \bn_\Gamma^T (\nablaG \bw)^T$ and $e_i$ the $i$-th basis vector in $\Bbb{R}^d$. This relation can be written as
$
   \dot{(\uD_i g)}=\uD_i \dot{g} - e_i^T \bA \nablaG g$.
For a vector function $\mathbf{g} =(g_1, \ldots, g_d)^T$ this  yields
$
 \dot{( \div_{\Gamma} \mathbf{g})} =\div_{\Gamma} \dot{\mathbf{g}} - {\rm tr} (\bA \nablaG \mathbf{g}).
$
For a scalar function $g$ the relation yields
$
 \dot{(\nablaG g)}= \nablaG \dot g - \bA \nablaG g.
$
Taking $\bg=\nablaG f$ thus results in the following relation:
\begin{equation} \label{commu2}
 \dot{(\Delta_\Gamma g)}-\Delta_\Gamma \dot g= - \div_\Gamma (\bA \nablaG g)- {\rm tr} (\bA \nablaG^2 g) =:R(\bw,g).
\end{equation}
We take $\phi=\phi_i$ ($1 \leq i \leq N$) in  \eqref{ODE}. Recall that from $f \in H^1(\Gs)$ and smoothness of $\Gs$  it follows that for $b,M,A$ in \eqref{ODE1} we have $b \in H^1(0,T;\Bbb{R}^N)$ and $M,A \in W^{1}_\infty(0,T;\Bbb{R}^{N \times N})$ and thus $\bu_N \in H^2(0,T;\Bbb{R}^N)$. Hence, differentiation w.r.t. $t$ of \eqref{ODE}, with  $\phi=\phi_i$, is allowed and using the Leibnitz formula, $\dot \phi_i=0$ and the commutation relation \eqref{commu2}  we obtain, with $v_N:=\dot u_N$,
\begin{equation}\label{eq_vN}
\begin{split}
 & \int_{\Gamma(t)} (\dot{v}_N -\Delta_{\Gamma} v_N)\phi_i \,d s \\  & =
-\int_{\Gamma(t)} (\dot{u}_N -\Delta_{\Gamma} u_N)\phi_i\div_{\Gamma}\bw \,d s + \int_{\Gamma(t)}( \dot{f}+f\div_{\Gamma}\bw +R(\bw,u_N)) \phi_i\,d s.
\end{split}
\end{equation}
We multiply this equation by $\dot u_i(t)$ and sum over $i$ to get
\begin{align}
 & \frac12 \frac{d}{dt} \int_{\Gamma(t)} v_N^2 \, \,d s +\int_{\Gamma(t)} |\nabla_{\Gamma} v_N|^2 \, \,d s  \label{er} \\
 &  = -\int_{\Gamma(t)} (\dot u_N -\Delta_{\Gamma} u_N)v_N\div_{\Gamma}\bw \,d s + \int_{\Gamma(t)} (\dot{f}+f\div_{\Gamma}\bw+R(\bw,u_N)) v_N\,d s \nonumber \\
 &\quad+ \frac12 \int_{\Gamma(t)} v_N^2\div_{\Gamma}\bw \, \,d s. \nonumber
\end{align}
To treat the first term on the right-hand side, we apply integration by parts and the Cauchy inequality:
\begin{align*}
 & |\int_{\Gamma(t)} (\dot{u}_N -\Delta_{\Gamma} u_N) v_N\div_{\Gamma}\bw \,d s|
\\
&\le c(\|\dot{u}_N\|_{L^2(\Gamma(t))}^2+ \|\nabla_{\Gamma} u_N\|_{L^2(\Gamma(t))}^2)+\frac14\|\nabla_{\Gamma} v_N\|_{L^2(\Gamma(t))}^2.
\end{align*}
For the second term we eliminate the second derivatives of $u_N$ that occur in $R(\bw, u_N)$ using the partial integration identity
$
  \int_\Gamma f \uD_i^2 g \, \,d s = - \int_\Gamma \uD_i f \uD_i g \, \,d s + \int_\Gamma f \uD_i g \kappa \nu_i \, \,d s
$.
Thus we get
\begin{align*}
 & |\int_{\Gamma(t)} (\dot{f}+f\div_{\Gamma}\bw+ R(\bw,u_N)) v_N\,d s| \\
& \le c(\|\dot f\|_{L^2(\Gamma(t))}+ \|f\|_{L^2(\Gamma(t))})\|v_N\|_{L^2(\Gamma(t))} +c \|u_N\|_{H^1(\Gamma(t))} \|v_N\|_{H^1(\Gamma(t))} \\
& \leq c \big( \|\dot f\|_{L^2(\Gamma(t))}^2+ \|f\|_{L^2(\Gamma(t))}^2 + \|u_N\|_{H^1(\Gamma(t))}^2 +\|\dot u_N\|_{L^2(\Gamma(t))}^2 \big)+ \frac14 \|\nabla_{\Gamma} v_N\|_{L^2(\Gamma(t))}^2.
\end{align*}
The two terms $\frac14\|\nabla_{\Gamma} v_N\|_{L^2(\Gamma(t))}^2$ can be absorbed by the term $\|\nabla_{\Gamma} v_N\|_{L^2(\Gamma(t))}^2$ on the left-hand side in \eqref{er}.
Using the estimates \eqref{est_energy2}, \eqref{uniH1} and a Gronwall inequality, we obtain from \eqref{er}
\begin{equation}\label{Aaux1}
\sup_{t\in(0,T)}\int_{\Gamma(t)}{v}_N^2\,d s + \int_{0}^T\int_{\Gamma(t)}|\nabla_\Gamma v_N|^2\,d s\,d t
\le C (\|f\|^2_{H^1(\Gs)}+\|v_N\|^2_{\Gamma_0}).
\end{equation}
Since $\bu_N \in H^2(0,T;\Bbb{R}^N)$, the function $\frac{d \bu_N}{dt}$ is continuous  and from \eqref{ODE1} we get
$\frac{d \bu_N}{dt}(0) = M(0)^{-1} b(0)=0$, 
due to the assumption $f(\cdot,0)=0$ on $\Gamma_0$.
Therefore, $v_N(x,0)= \sum_{j=1}^N \frac{d \bu_j}{dt}(0) \phi_j(x,0)=0$ on $\Gamma_0$.
Using this in \eqref{Aaux1} we get
\begin{equation}\label{Aaux11}
 \sup_{t\in[0,T]}\int_{\Gamma(t)} v_N^2\, dt +\|v_N\|_H^2= \sup_{t\in[0,T]}\int_{\Gamma(t)} \dot u_N^2\, dt + \|\dot u_N\|_H^2 \leq C \|f\|_{H^1(\Gs)}^2
\end{equation}
uniformly in $N$. Hence for a subsequence, again denoted by $(v_N)_{N \in \Bbb{N}}$, we have  $v_N \rightharpoonup  v$ in $H$. This implies, cf. \eqref{elFA}, $v_N \rightharpoonup  v$ in $L^2(\Gs)$. Due to \eqref{conv2} and uniqueness of weak limits we obtain $v=\dot u$, i.e.
\begin{equation} \label{rr}
  v_N \rightharpoonup  \dot u \quad \text{in}~~H
\end{equation}
holds.
 Passing to the limit in \eqref{Aaux11} yields, cf. exercise~7.5.5 in~\cite{Evans},  \[\sup_{t\in[0,T]}\int_{\Gamma(t)} \dot u^2\, dt +\|\dot u\|_H \leq C \|f\|_{H^1(\Gs)},\] which implies
\begin{equation}\label{uest3}
\|\nablaG \dot u\|_0 \le C \|f\|_{H^1(\Gs)}
\end{equation}
and by \eqref{eqAubin} it also implies
\begin{equation}\label{uest3b}
\sup_{t\in[0,T]}\| u\|_{H^2(\Gamma(t))}\le C \|f\|_{H^1(\Gs)}.
\end{equation}
\emph{6. The estimate  $\|\ddot u\|_0 \leq c \|f\|_{H^1(\Gs)}$ holds.}
 First we show $\ddot{u}\in H'$. For arbitrary $\zeta\in C^1(\Gs)$ and $\zeta_N=P_{X_N(t)}\zeta(\cdot,t) \in X_N(t)$, with  $P_{X_N(t)}$ the orthogonal projection defined in Lemma~\ref{l_A2}, using the relation \eqref{eq_vN}  we obtain
\[
\begin{split}
 & \la\ddot{u}_N,\zeta\ra
=\int_0^T\int_{\Gamma(t)}\ddot{u}_N\zeta \, \,d s \,d t=\int_0^T\int_{\Gamma(t)}\ddot{u}_N\zeta_N\,d s \,d t=\int_0^T\int_{\Gamma(t)}\dot{v}_N\zeta_N\,d s \,d t\\
&=\int_0^T\int_{\Gamma(t)}[(\dot{f} + \Delta_{\Gamma} v_N)
-(\dot{u}_N -\Delta_{\Gamma} u_N)\div_{\Gamma}\bw + f\div_{\Gamma}\bw+R(\bw,u_N)] \zeta_N\,d s \,d t.
\end{split}
\]
Applying integration by parts, the Cauchy inequality, Lemma~\ref{l_A2} and the estimates \eqref{est_energy2} and \eqref{Aaux1}, we get
\[
|\la\ddot{u}_N,\zeta\ra | \le c\,\|f\|_{H^1(\Gs)}\left(\int_0^T\|\zeta_N\|_{L^2(\Gamma(t))}^2+\|\nabla_{\Gamma} \zeta_N\|_{L^2(\Gamma(t))}^2\,d t\right)^{\frac12}\le
c\,\|f\|_{H^1(\Gs)}\|\zeta\|_{H}.
\]
Since $C^1(\Gs)$ is  dense in $H$, we get $\ddot{u}_N\in H'$ and $\|\ddot{u}_N\|_{H'}\le c\, \|f\|_{H^1(\Gs)}$, uniformly in $N$. Take $\zeta \in C^1_0(\Gs)$. Recall that $\dot u_N\rightharpoonup  \dot u$ in  $L^2(\Gs)$, cf. \eqref{conv2}. Using this we get
\begin{align*}
\la \ddot{u},\zeta\ra & :=-\int_0^T\int_{\Gamma(t)}\dot{u}\dot{\zeta}+\dot{u}\zeta\div_{\Gamma}\bw \, \,d s \,d t
=-\lim_{N\to\infty}\int_0^T\int_{\Gamma(t)}\dot{u}_N\dot{\zeta}+\dot{u}_N\zeta\div_{\Gamma}\bw \, \,d s \,d t
\\  & =\lim_{N\to\infty}\la \ddot{u}_N,\zeta\ra\le
\sup_N\|\ddot{u}_N\|_{H'}\|\zeta\|_H\le c\|f\|_{H^1(\Gs)}\|\zeta\|_H.
\end{align*}
Therefore,   $\ddot{u}\in H'$ and $\|\ddot{u}\|_{H'}\le c\, \|f\|_{H^1(\Gs)}$ and $\ddot u_N \rightharpoonup \ddot u$ in $H'$. Thus, for $v_N=\dot u_N$, $v=\dot u$ we have, cf. \eqref{rr},
\begin{equation} \label{ppp}
  v_N \rightharpoonup v \quad \text{in}~~H, ~~  \dot v_N \rightharpoonup \dot v \quad \text{in}~~H'.
\end{equation}
We take test function $\psi(x,t)=t^j \phi_k(x,t)$ as in step 3. Using the relation \eqref{eq_vN}, we get for $N \geq k$:
\begin{multline*}
\la \dot v_N,\psi \ra +(\nablaG v_N,\nablaG \psi)_0= (\dot f + R(\bw,u_N),\psi)_0 \\
  - \big[(\dot u_N, \psi \div_{\Gamma}\bw)_0 + (\nablaG u_N,\nablaG (\psi \div_{\Gamma}\bw))_0 -(f,\psi \div_{\Gamma}\bw)\big].
\end{multline*}
For $N \to \infty$, due to $ u_N \rightharpoonup u$ in $H^1(\Gs)$, we can replace $u_N$ by $u$ and since $u$ is the solution of \eqref{weakformuP} the term between square brackets vanishes. Using the weak limit results in \eqref{ppp} and applying a density argument (as in step 3) we thus obtain
\[
\la\dot{v}, \xi\ra+(\nablaG v,\nablaG \xi)_0=(\dot{f}+ R(\bw,u),\xi)_0\quad \text{for all}~\xi\in H.
\]
 From $v_N \rightharpoonup v$ in $W$, boundedness of the trace operator from $W$ to $L^2(\Gamma_0)$ we obtain $v_N(\cdot,0) \rightharpoonup v(\cdot,0)$ in $L^2(\Gamma_0)$. Hence, due to  $v_N|_{\Gamma_0}=0$ we obtain $v|_{\Gamma_0}=0$.
Therefore, for the function $v:=\dot u$, we have $v\in W_0$ is the weak solution of the surface parabolic equation \eqref{weakformuP} with the right hand side $f^\ast=\dot f+R(\bw,u)$ from $L^2(\Gs)$. Hence we can apply the regularity result in \eqref{uest1}
and get $\dot{v}\in L^2(\Gs)$. Thus, $\ddot{u}\in L^2(\Gs)$ and $\|\ddot{u}\|_0\le  C \|f^\ast\|_0 \leq \|\dot f\|_0+ \big(\int_0^T \|u\|_{H^2(\Gamma(t))}^2 \, \,d t\big)^\frac12 \leq  C \|f\|_{H^1(\Gs)}$.
Finally note that from this estimate and the results in \eqref{fg}, \eqref{uest3}, \eqref{uest3b} we obtain the $H^2$-regularity estimate in \eqref{H2reg}.

\section{Conclusions and outlook} \label{sectconcl} We analyzed an Eulerian method based on
traces on the space-time manifold of standard bilinear space-time finite elements. A stability result is derived in which there are no restrictions on the size of $\Delta t$ and $h$. This indicates that the method has favourable robustness properties.
We proved first and  second order discretization error bounds for this method. To the best of our knowledge, this is the first  Eulerian finite element method
which is proved to be second order accurate for PDEs on  evolving surfaces.
In the applications that we consider, we restrict to first order finite elements, due to the fact that the approximation of the evolving surface causes an error (``geometric error'') of size $\mathcal{O}(h^2)$, which is consistent with the interpolation error for P1 elements. Results of numerical experiments, which illustrate the second order convergence and excellent stability properties of the method,  are presented in \cite{refJoerg,ORXsinum,GOReccomas}. These  experiments clearly indicate
that second order convergence holds in $L^2(\Gs)$ norm, which is  stronger than the $H^{-1}(\Gs)$ norm  used in our analysis. The experiments also show  that the stabilization term ($\sigma >0$ in \eqref{stabblf})
improves the discrete mass conservation of the method, but is not  essential for stability or overall accuracy. Essential for our analysis is the condition~\eqref{ass7}, which allows a  strong convection of $\Gamma(t)$ but only small local area changes. Numerical experiments indicate that the latter is not  critical for the performance of the method.

There are several topics that we consider to be of interest for further research. Maybe an error analysis that needs weaker assumptions (than \eqref{ass7}) and/or avoids the stabilization can be developed. A second interesting topic is the derivation  of anisotropic interpolation error estimates which may then lead to a second order error bound in the $L^2(\Gs)$ norm.
A further open problem is the derivation of rigorous error estimates for the case when the smooth space-time manifold $\Gs$ is approximated, e.g., by a piecewise tetrahedral  surface.
\bibliographystyle{siam}
\bibliography{literatur}

\begin{thebibliography}{10}

\bibitem{Aubin}
{\sc T.~Aubin}, {\em Nonlinear analysis on manifolds. Monge-Ampere equations},
  Springer, Berlin, 1982.

\bibitem{BrennerScott}
{\sc L.~Brenner, S.and~Scott}, {\em The Mathematical Theory of Finite Element
  Methods}, Springer, New York, second~ed., 2002.

\bibitem{GrainBnd1}
{\sc J.~W. Cahn, P.~Fife, and O.~Penrose}, {\em A phase field model for
  diffusion induced grain boundary motion}, Acta Mater, 45 (1997),
  pp.~4397--4413.

\bibitem{Demlow06}
{\sc A.~Demlow and G.~Dziuk}, {\em An adaptive finite element method for the
  {Laplace-Beltrami} operator on implicitly defined surfaces}, SIAM J. Numer.
  Anal., 45 (2007), pp.~421--442.

\bibitem{Dziuk07}
{\sc G.~Dziuk and C.~Elliott}, {\em Finite elements on evolving surfaces}, IMA
  J. Numer. Anal., 27 (2007), pp.~262--292.

\bibitem{DziukElliot2010}
\leavevmode\vrule height 2pt depth -1.6pt width 23pt, {\em An {Eulerian}
  approach to transport and diffusion on evolving implicit surfaces}, Comput.
  Vis. Sci., 13 (2010), pp.~17–--28.

\bibitem{DEreview}
{\sc G.~Dziuk and C.~M. Elliott}, {\em Finite element methods for surface
  {PDEs}}, Acta Numerica, 22 (2013), pp.~289--396.

\bibitem{DziukElliot2013a}
{\sc G.~Dziuk and C.~M. Elliott}, {\em $l^2$-estimates for the evolving surface
  finite element method}, Mathematics of Computation, 82 (2013), pp.~1--24.

\bibitem{DziukKroener}
{\sc G.~Dziuk, D.~Kr\"oner, and T.~M\"uller}, {\em Scalar conservation laws on
  moving hypersurfaces}, preprint, http://aam.uni-freiburg.de/abtlg/ls/lskr,
  Department of Applied Mathematics, University of Freiburg, 2012.

\bibitem{ElliotStinner}
{\sc C.~M. Elliott and B.~Stinner}, {\em Modeling and computation of two phase
  geometric biomembranes using surface finite elements}, Journal of
  Computational Physics, 226 (2007), pp.~1271--1290.

\bibitem{ErikJohn}
{\sc K.~Eriksson and C.~Johnson}, {\em Adaptive finite element methods for
  parabolic problems {I}: A linear model problem}, SIAM Journal on Numerical
  Analysis, 28 (1991), pp.~43--77.

\bibitem{Ern04}
{\sc A.~Ern and J.-L. Guermond}, {\em Theory and practice of finite elements},
  Springer, New York, 2004.

\bibitem{Evans}
{\sc L.~Evans}, {\em Partial Differential Equations}, AMS, 1998.

\bibitem{GT}
{\sc D.~Gilbarg and N.~Trudinger}, {\em Elliptic Partial Differential Equations
  of Second Order}, Springer, New York, 2001.

\bibitem{refJoerg}
{\sc J.~Grande}, {\em Finite element methods for parabolic equations on moving
  surfaces}, Preprint 360, IGPM RWTH Aachen University. Accepted for
  publication in SIAM J. Sci. Comp., 2013.

\bibitem{GOReccomas}
{\sc J.~Grande, M.~Olshanskii, and A.~Reusken}, {\em A space-time {FEM} for
  {PDEs} on evolving surfaces}, in proceedings of 11th World Congress on
  Computational Mechanics, E.~Onate, J.~Oliver, and A.~Huerta, eds., Eccomas.
  IGPM report 386 RWTH Aachen, 2014.

\bibitem{GReusken2011}
{\sc S.~{Gro\ss} and A.~Reusken}, {\em Numerical Methods for Two-phase
  Incompressible Flows}, Springer, Berlin, 2011.

\bibitem{Hansbo02}
{\sc A.~Hansbo and P.~Hansbo}, {\em An unfitted finite element method, based on
  {Nitsche's} method, for elliptic interface problems}, Comput. Methods Appl.
  Mech. Engrg., 191 (2002), pp.~5537--5552.

\bibitem{Hansbo03}
{\sc A.~Hansbo, P.~Hansbo, and M.~Larson}, {\em A finite element method on
  composite grids based on {N}itsche's method}, Math. Model. Numer. Anal., 37
  (2003), pp.~495--514.

\bibitem{Hunter}
{\sc J.~Hunter}, {\em Notes on partial differential equations}, {Lecture
  Notes}, www.math.ucdavis.edu/hunter/pdes/pdes.html, Dept. Math., Univ. of
  California.

\bibitem{James04}
{\sc A.~James and J.~Lowengrub}, {\em A surfactant-conserving volume-of-fluid
  method for interfacial flows with insoluble surfactant}, J. Comp. Phys., 201
  (2004), pp.~685--722.

\bibitem{GrainBnd2}
{\sc U.~F. Mayer and G.~Simonnett}, {\em Classical solutions for diffusion
  induced grain boundary motion}, J. Math. Anal., 234 (1999), pp.~660--674.

\bibitem{Novaketal}
{\sc I.~L. Novak, F.~Gao, Y.-S. Choi, D.~Resasco, J.~C. Schaff, and
  B.~Slepchenko}, {\em Diffusion on a curved surface coupled to diffusion in
  the volume: application to cell biology}, Journal of Computational Physics,
  229 (2010), pp.~6585--6612.

\bibitem{OlshanskiiReusken08}
{\sc M.~Olshanskii and A.~Reusken}, {\em A finite element method for surface
  {PDEs}: matrix properties}, Numer. Math., 114 (2009), pp.~491--520.

\bibitem{OlshReusken08}
{\sc M.~Olshanskii, A.~Reusken, and J.~Grande}, {\em A finite element method
  for elliptic equations on surfaces}, SIAM J. Numer. Anal., 47 (2009),
  pp.~3339--3358.

\bibitem{ORXsinum}
{\sc M.~Olshanskii, A.~Reusken, and X.~Xu}, {\em An {Eulerian} space-time
  finite element method for diffusion problems on evolving surfaces}, NA\&SC
  Preprint No 5, Department of Mathematics, University of Houston. Accepted for
  publication in SIAM J. Numer. Anal.,  (2013).

\bibitem{Stone}
{\sc H.~Stone}, {\em A simple derivation of the time-dependent
  convective-diffusion equation for surfactant transport along a deforming
  interface}, Phys. Fluids A, 2 (1990), pp.~111--112.

\bibitem{XuZh}
{\sc J.-J. Xu and H.-K. Zhao}, {\em An {Eulerian} formulation for solving
  partial differential equations along a moving interface}, Journal of
  Scientific Computing, 19 (2003), pp.~573--594.

\end{thebibliography}

\end{document}